\documentclass[3p,times]{elsarticle}

 \usepackage{lipsum}
\usepackage{color}
\usepackage{amsfonts}
\usepackage[leqno]{amsmath}
\usepackage{amsthm}
\usepackage{graphicx}
\usepackage{epstopdf}
\usepackage{algorithm}
\usepackage[algo2e]{algorithm2e}
\usepackage{subcaption,caption}
\usepackage{float} 
\usepackage{multirow}
\usepackage{footmisc} 
\usepackage{lineno}
\RequirePackage{xcolor}[2.11]
\usepackage[colorlinks=true, linkcolor=cyan,
    		   citecolor=red!60!black, urlcolor=cyan]{hyperref}
\numberwithin{equation}{section}

    \newtheorem{theorem}{Theorem}[section]
    \newtheorem{definition}{Definition}[section]    
    \newtheorem{lemma}{Lemma}[section]
    \newtheorem{example}{Example}[section]
    \newtheorem{remark}{Remark}[section]
	\newtheorem{corollary}{Corollary}[section]

\graphicspath{ {./Figure/} }

\def\R{\mathbb{R}}
\def\bx{{\bf x}}

\def\bz{{\bf z}}
\def\bu{{\bf u}}

\def\bv{{\bf v}}
\def\bw{{\bf w}}
\def\bb{{\bf b }}
\def\ba{{\bf a }}

\def\A{{\bf A }}
\def\H{{\bf H }}
\def\be{{\boldsymbol\eta }}
\def\bg{{\boldsymbol g }}
\def\supp{{\rm supp }}

\begin{document}

\begin{frontmatter}




\title{ Gradient Projection Newton Pursuit for Sparsity Constrained Optimization}


\author{Shenglong Zhou}

\address{Department of EEE, Imperial College London,  United Kingdom. (Email: shenglong.zhou@imperial.ac.uk)\vspace{-5mm}}
 
\begin{abstract}
Hard-thresholding-based algorithms have seen various advantages for sparse optimization in controlling the sparsity and allowing for fast computation. Recent research shows that when techniques of the Newton-type methods are integrated, their numerical performance can be improved surprisingly. This paper develops a gradient projection Newton pursuit algorithm that mainly adopts the hard-thresholding operator and employs the Newton pursuit only when certain conditions are satisfied. The proposed algorithm is capable of converging globally and quadratically under the standard assumptions. When it comes to compressive sensing problems, the imposed assumptions are much weaker than those for many state-of-the-art algorithms. Moreover, extensive numerical experiments have demonstrated its high performance in comparison with the other leading solvers.
\end{abstract}


\end{frontmatter}


\section{Introduction}
This paper focuses on the following sparsity constrained optimization (SCO):
\begin{equation}
\min_{\bx}~f(\bx), ~~{\rm s.t.}~ \|\bx\|_{0}\leq s,\tag{SCO}\label{sco} 
\end{equation} 
where $f:\R^n\mapsto\R$ is twice continuously differentiable and bounded from below and $\|\bx\|_0$ is the zero norm counting the number of nonzero entries of $\bx$,  and $s \ll n$ is a given integer.   A typical example of \eqref{sco} is compressive sensing (CS) which has shown revolutionary advances both in  theory and algorithms over the last few decades. We refer to ground-breaking papers  \cite{candes2005decoding, donoho2006compressed, candes2006robust},   surveys \cite{bruckstein2009sparse,rauhut2010compressive}, and  monographs  \cite{eldar2012compressed,foucart2017mathematical,parikh2014proximal} for more information.  Other examples of SCO include sparse logistic regression \cite{bahmani2013greedy, wangaextended}, quadratic compressive sensing \cite{shechtman2011sparsity,Beck13}, and phase retrieval \cite{shechtman2014gespar}.

\subsection{Hard-thresholding based algorithms}

There is an impressive body of work on developing numerical algorithms to solve SCO problems. We opt to conduct a bit of technical review on a small number of papers based on the hard-thresholding principle that directly motivate our research. Those reviewed papers more or less suggest the algorithmic framework described in Algorithm \ref{algorithm 0},  
where ${\rm supp}(\bx)$ is the support set of $\bx$, namely the set of indices of its nonzero elements and $ \Pi_{s}(\cdot)$ is the so-called hard-thresholding operator defined by  \begin{eqnarray}\label{htp-operator}
\Pi_{{s} }(\bx):= {\rm argmin}\{\|\bw-\bx\|: \|\bw\|_0 \leq {s} \},
\end{eqnarray}
where $\|\cdot\|$ is the Euclidean norm.  The operator can be derived by keeping the $s$ largest elements in the magnitude of $\bx$ and setting the remaining to zeros.  We note that problem \eqref{htp-operator} may have multiple solutions. As a result, operator $\Pi_{{s} }(\bx)$ is a set that may contain multiple elements as well.   Here,  $\phi(\bx^k)=\nabla f(\bx^k)$ or $\bx^k+ \mu \nabla f(\bx^k)$  with $\mu>0$ and  $T_k=\supp(\bu^{k+1})$ or $\supp(\bu^{k+1}) \cup \supp(\bx^k)$. Their choices for different algorithms  can be found in Table \ref{algs-framework}, where we present seven algorithms whose frameworks fall into Algorithm \ref{algorithm 0}.  For instance, for  CS problems, IHT \cite{blumensath2009iterative} takes the hard-thresholding operator as the next point directly (i.e., $\bx^{k+1}=\bu^{k+1}$). HTP \cite{foucart2011hard} first calculates  the hard-thresholding operator to derive $\bu^{k+1}$  and then updates  next point $\bx^{k+1}=\bv^{k+1}$ via solving Step 2 on a subspace decided by  $\supp(\bu^{k+1})$. CoSaMP and SP share similar patterns that perform all the three steps, where $T_k$ is chosen to be the union of the support sets of the previous point (i.e. $\bx^{k}$) and a point (i.e. $\bu^{k+1}$) from the hard-thresholding operator. In Step 3, they prune the point so as to ensure point $\bz^{k+1}$ to be $s$-sparse, namely, $\|\bz^{k+1}\|_0\leq s$. When it comes to general SCO problems,  algorithms including IHT \cite{Beck13}, GraHTP  \cite{yuan2017gradient} and  GraSP  \cite{bahmani2013greedy} can be reduced to the frameworks of IHT, HTP  and CoSaMP, respectively. 
 
 \begin{algorithm}[!th]
\SetAlgoLined
 Initialize $\bx^{0}$, set $k\Leftarrow0$.\\
 \While{{The halting condition is not met}}{
$
 \arraycolsep=1.4pt\def\arraystretch{1.25}
\begin{array}{llll}
\emph{Step 1. Hard-thresholding:} &\bu^{k+1} \in \Pi_{r}( \phi(\bx^k)),\\
\emph{Step 2. Debiasing:} &\bv^{k+1}\in{\rm argmin}_{\bx}\{f(\bx): \supp(\bx)\subseteq T_k \}, \\
\emph{Step 3. Pruning:} & \bz^{k+1}\in\Pi_{s}(\bv^{k+1}).\\  
\end{array}
$\\
Update $\bx^{k+1}\in\{\bu^{k+1}, \bv^{k+1}, \bz^{k+1}\}$ and set $k:= k+1$.
 }
 Output the solution $\bx^k.$
 \caption{Hard-thresholding-based algorithms \label{algorithm 0}}
\end{algorithm}
  
\begin{table}[H]
	\renewcommand{\arraystretch}{1.0}\addtolength{\tabcolsep}{9pt}
	\caption{Descriptions of hard-thresholding-based algorithms.}\vspace{-2mm}
	\label{algs-framework}
	\begin{center}
		\begin{tabular}{lccccccccc}
			\hline
			Algs. & Ref. & $\phi (\bx^k)$ & $r$ & $T_k$& Steps & $\bx^{k+1}$  \\\hline
\multicolumn{7}{c}{CS problems}\\\hline
IHT& \cite{blumensath2009iterative}& \multirow{1}{*}{$\bx^k+ \mu \nabla f_{cs}(\bx^k)$} &  \multirow{1}{*}{ $ s$}& &\multirow{1}{*}{1} &\multirow{1}{*}{$\bu^{k+1}$} \\
HTP& \cite{foucart2011hard}& {$\bx^k+ \mu \nabla f_{cs}(\bx^k)$}&  { $s$}& $\supp(\bu^{k+1})$  &  1, 2& {$\bv^{k+1}$} \\
CoSaMP& \cite{needell2009cosamp}&  {$\nabla f_{cs}(\bx^k)$}&  { $2s$}& $\supp(\bu^{k+1}) \cup \supp(\bx^k)$  &  1, 2 ,3& {$\bz^{k+1}$} \\

SP& \cite{dai2009subspace}&  {$\nabla f_{cs}(\bx^k)$}&  { $s$}& $\supp(\bu^{k+1}) \cup \supp(\bx^k)$  &  1, 2, 3& {$\bz^{k+1}$} \\
			\hline
\multicolumn{7}{c}{General SCO problems}\\\hline
IHT& \cite{Beck13}& {$\bx^k+ \mu \nabla f (\bx^k)$} &  { $ s$}& & {1} & {$\bu^{k+1}$} \\			
GraHTP& \cite{yuan2017gradient} & {$\bx^k+ \mu \nabla f(\bx^k)$}&  { $s$}& $\supp(\bu^{k+1})$  &  1, 2& {$\bv^{k+1}$} \\
GraSP& \cite{bahmani2013greedy}&  {$\nabla f (\bx^k)$}&  { $2s$}& $\supp(\bu^{k+1}) \cup \supp(\bx^k)$  &  1, 2 ,3& {$\bz^{k+1}$} \\\hline
		\end{tabular}
	\end{center}
\end{table}

For  CS problems,  the objective function is
\begin{equation*}
f(\bx):= f_{cs}(\bx):=({1}/{2} )\|\A\bx-\bb\|^{2},
\end{equation*} 
where $\A\in\R^{m\times n}$ and $\bb\in\R^m$ are the sensing matrix  and the observation vector. Hard-thresholding-based algorithms have been shown to converge an arbitrarily given $s$-sparse signal  $\bx^*$  under the restricted isometry property (RIP) of sensing matrix $\A$. Generally speaking, the theory ensures that the distance between each iterate to the sparse signal is bounded by the sum of two terms.  
The first term converges linearly, and the second term is a fixed approximation error that depends on the noise. We refer to literature \cite{blumensath2009iterative,needell2009cosamp,foucart2011hard,zhao2020improved} for many of such a result, which is often called statistical error-bound guarantee \cite{shen2017tight}. To guarantee the error-bound theory, there is an extensive usage of the concept of RIP of the sensing matrix associated with a restricted
isometry constant (RIC) of order $s$, denoted by $\delta_s$ \cite{candes2005decoding,candes2006near}, such as conditions $\delta_{3s}<0.177$ \cite{blumensath2009iterative}, $\delta_{3s}<0.618$ \cite{zhao2020improved} for IHT, $\delta_{3s}<0.577$ \cite{foucart2011hard} for HTP, $\delta_{3s}<0.165$ \cite{dai2009subspace} for SP,  and $\delta_{4s}<0.100$ \cite{needell2009cosamp}, $\delta_{4s}<0.478$ \cite{foucart2017mathematical}, $\delta_{4s}<0.510$  \cite{zhao2020improved} for CoSaMP.  A separate line of establishing the convergence property for hard-thresholding-based algorithms for  CS problems is from the optimization perspective, which  contrasts to the statistical error bound theory. It has been shown in \cite{blumensath2008iterative} that IHT converges to a local minimizer of SCO provided that $\A$ is full row rank and its spectral norm is smaller than one. A similar result was also achieved for NIHT if $\A$ is full row rank and $s$-regular \cite{blumensath2010normalized}, {where $\A$ is $s$-regular if any $s$ columns are linearly independent (see \cite[Definition 2.2]{Beck13})}. Then the authors in \cite{Beck13} proved that the whole sequence generated by IHT converges to an $L$-stationary point (see Definition \ref{def3.1}) if $\A$ is $s$-regular. 

For general SCO problems, a couple of convergence results from the optimization and statistics perspectives have been achieved. For starters, authors in \cite{Beck13} investigated IHT thoroughly and proved that any accumulating point of the sequence of IHT is an $L$-stationary point if the objective function is gradient Lipschitz continuous (see \cite[Theorem 3.1]{Beck13}). Almost at the same time,  authors in \cite{bahmani2013greedy} showed that the sequence generated by GraSP has statistical error-bound under the assumptions of regularities (see Definition \ref{def-RSCS}). Such an error-bound theory was also built for GraHTP in \cite{yuan2017gradient}. However, these methods have not seen better results than the linear convergence rate. Very recently, authors in \cite{zhou2021global} developed a Newton hard-thresholding pursuit (NHTP) algorithm with an excellent numerical performance in comparison with a number of leading solvers. The  algorithm has been proven to have a global convergence property and a quadratic convergence rate under the assumptions of regularities. However, the convergence results were obtained through carefully justifying parameters, which somewhat made the proofs very technical and restricted the assumptions of the problem. More comments are provided in Remarks \ref{remark-nhtp-gpnp} and \ref{remark-ass-cs}.

\subsection{Contributions}
 The aim of this paper is to develop a gradient projection Newton pursuit (GPNP) algorithm that possesses strong theoretical and numerical advantages. We summarize the main contributions as follows. 
 
{\bf I.  Distinct algorithmic framework.} We note that IHT or NIHT only performs the hard-thresholding operator to update the next point and hence belongs to the first-order methods with low computational complexity but  slow convergence rates. When the debiasing step is adopted in HTP or Newton step is used in NHTP, it would significantly accelerate the convergence and enhance the recovery accuracies. However, HTP might be too greedy since it does the debiasing every step even when $T_k$ is not identical to the support of the true signal.  Combining these two aspects, as shown in Algorithm \ref{Newton-switch-on}, GPNP takes hard-thresholding as the main step while performing Newton steps only when certain conditions (i.e., \eqref{Newton-switch-on}) are met.  Such a framework not only leads to a low computational complexity but also overcomes the drawback of HTP (i.e., being too greedy).

{\bf II.  Weaker assumptions for convergence analysis.} Similar to \cite{ Beck13, zhou2021global}, we also conduct convergence analysis for GPNP from the perspective of optimization. We prove that the generated sequence converges to a unique local minimizer globally and quadratically under the assumptions of regularities of the objective function. When it comes to CS problems, such convergence properties can be preserved {if $\A$ is $s$-regular,} which is weaker than the sufficient conditions on RIC for some greedy methods, such as CoSaMP, SP, HTP, and NHTP. Furthermore, the quadratic convergence turns out to be the termination within finitely many steps, a more pleasant property.

 {\bf III.  High numerical performance.} The extensive numerical experiments have demonstrated that GPNP is capable of running quickly and delivering {relatively accurate solutions.} When compared with other state-of-the-art methods, it is able to outperform them. Taking CS problems as an example, we select nine greedy and nine relaxation methods for comparisons. In general, greedy methods have excellent performance in terms of rendering accurate recoveries and running fast. However, they are weaker than relaxation methods to ensure successful recoveries, which can be testified by the results in Figure  \ref{fig:succ-s-ex1}. For instance,  relaxation methods like IRL1, DCL1L2, IRLSLq, and SAMP obtain higher success rates than the other  greedy methods. Nevertheless, GPNP, as a greedy method, not only significantly improves the success rates  for the greedy family  but also outperforms the selected relaxation methods.  
 
 \vspace{3mm}
\begin{figure}[!th]
\begin{subfigure}{.495\textwidth}
	\centering
	\includegraphics[width=1\linewidth]{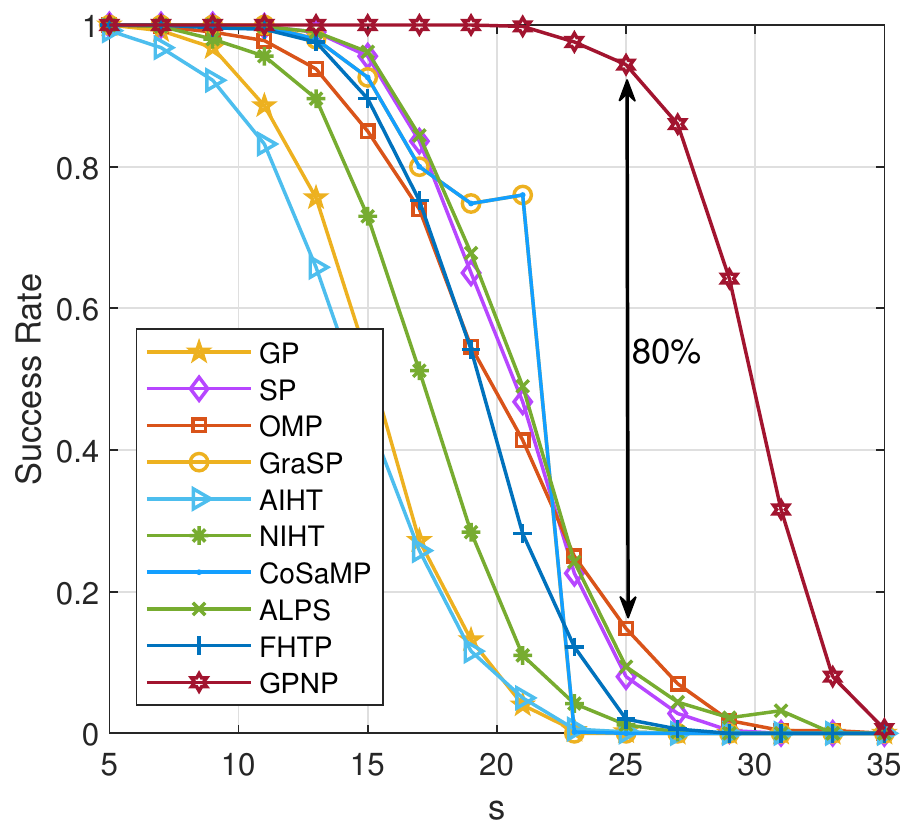}
	\caption{GPNP v.s. greedy methods.}
	\label{fig:succ-s-ex1-greedy}
\end{subfigure}	 
\begin{subfigure}{.495\textwidth}
	\centering
	\includegraphics[width=1\linewidth]{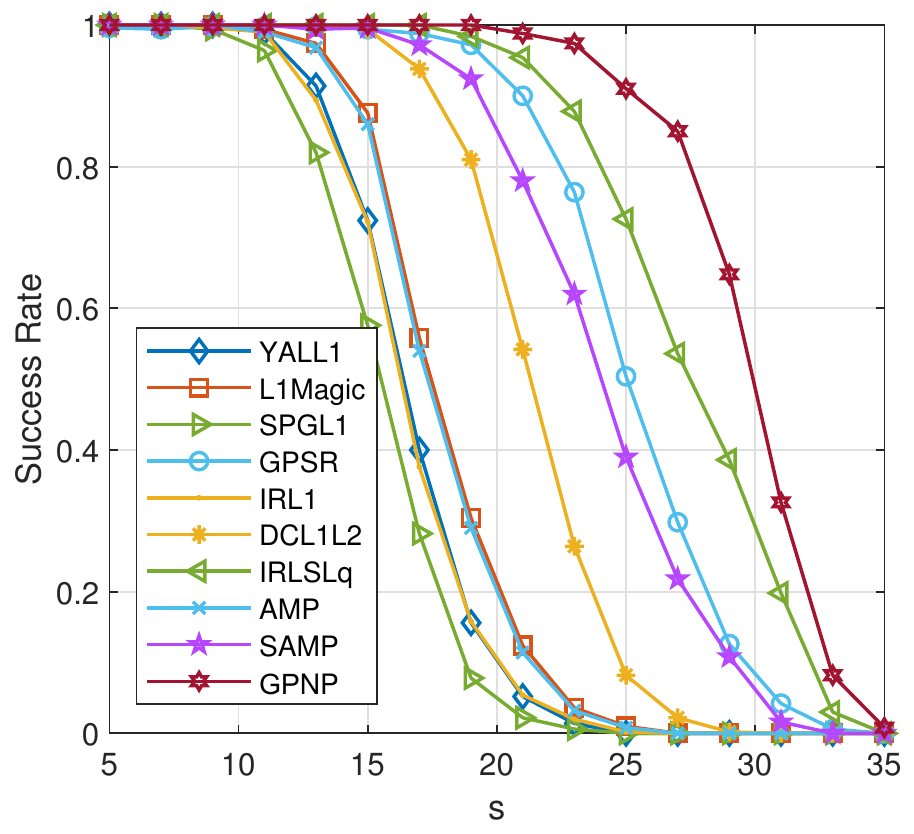}
	\caption{GPNP v.s. relaxation methods.}
	\label{fig:succ-s-ex1-non-greedy}
\end{subfigure}
\caption{Success rate v.s. sparsity level $s$ for Example \ref{cs-ex}.\label{fig:succ-s-ex1}}
\end{figure}

 \subsection{Organization and notation}
This paper is organized as follows.  In the next section, we present  the regularities of $f$ for convergence analysis as well as the optimality conditions of \eqref{sco}.  In  Section \ref{sec:htnp},  we  design the gradient projection Newton pursuit (GPNP) and establish its global and quadratic convergence properties, followed by the application into CS problems. In Section \ref{sec:num}, we aim to  demonstrate the high performance of GPNP by conducting extensive numerical comparisons among GPGN and a number of leading solvers for solving CS and QCS problems. {Some conclusive remarks are given in the last section. }
 
To end this section, apart from the aforementioned notation, we also summarize some other ones here. Throughout the paper, we denote $[m]:=\{1,2,\ldots ,m\}$ and $S$ the sparse set in $\R^n$ by 
$$S :=\{\bx\in \R^{n}:\|\bx \|_{0}\leq s\}.$$ 
Let $\|\cdot\|$ be the Euclidean norm for a vector and the spectral norm for a matrix.  The complement set of $T$ is written as $\overline{T}$. For vector $\bx$,  its neighbourhood with a positive radius $c$, support set, subvector indexed on $T$, and $s$th largest element in magnitude is written as $N(\bx,c):=\{\bw\in\R^n: \|\bx-\bw\|\leq c\}$,    $\supp(\bx):=\{i\in [n]: x_{i}\neq 0\}$,   $\bx_{T}$, and  $\bx^\downarrow_{s}$, respectively.  For matrix $\A\in\R^{m\times n}$,   $\A_{TJ}$ represents the sub-matrix containing rows indexed by ${T}$ and columns indexed by $J$, in particular, $\A_{T:}:=\A_{T[n]}$ and $\A_{:J}:=\A_{[m]J}$. Finally, we write the sub-gradient and  sub-Hessian  as
\begin{eqnarray*}
\nabla_T f(\bx):=(\nabla f(\bx))_{T},~~\nabla_{T,J}^2f(\bx):=(\nabla^2f(\bx))_{T,J}.
\end{eqnarray*}
\section{Preliminaries}\label{sec:pre}

The common assumptions on the objective function used for convergence analysis are regularities that often share the property of restricted  strong convexity/smoothness, see \cite{agarwal2012fast, shalev2010trading, negahban2012unified, bahmani2013greedy, blumensath2013compressed, yuan2017gradient,zhou2021global}.
We present them below in a way that is conducive to our technical proofs.

\begin{definition} \label{def-RSCS}(Restricted strongly convex and smooth functions)
Suppose that $f: \mathbb{R}^n \mapsto \mathbb{R}$ is twice continuously differentiable. {Let $r\in[n]$ be an integer (e.g., $r=s$ or $r=2s$ used in this paper). }
\begin{itemize}
	\item[(i)] We say $f$ is $r$-restricted strongly smooth ($r$-RSS) if there  exists  a constant $L_{r} >0$ such that, for any $s$-sparse vectors $\bz,\bx$ with $\|\bz-\bx\|_0\leq r$,
	\begin{eqnarray}\label{r-rss}
	f(\bz) \leq f(\bx)+\langle\nabla f(\bx), \bz-\bx\rangle+(L_r/{2})\|\bz-\bx\|^{2}.	
	\end{eqnarray}

	\item[(ii)] We say $f$ is $r$-restricted strongly convex ($r$-RSC) if there exists a constant $\ell_{r} >0$
	such that,  for any $s$-sparse vectors $\bz,\bx$ with $\|\bz-\bx\|_0\leq r$,
	\begin{eqnarray}\label{r-rsc}
	f(\bz) \geq f(\bx)+\langle\nabla f(\bx), \bz-\bx\rangle+(\ell_r/{2})\|\bz-\bx\|^{2}.	
	\end{eqnarray}
When   $\ell_{r}=0$,  $f$ is said to be $r$-restricted convex ($r$-RC). 
We say $f$ is locally $r$-RSC at $\bx$ if the above inequality only holds for those $s$-sparse vectors $\bz$ around $\bx$.
	
\item[(iii)] We say that $f$ is locally $r$-restricted Hessian Lipschitz continuous ($r$-RHLC) at a  $s$-sparse vector $\bx$ if
there exists a Lipschitz constant  $M_r>0$ (relied on $\bx$) such that, for any vectors  {$\bw,\bz\in S\cap N(\bx,\eta\}$} with $\|\bw-\bz\|_0\leq r$, 
\begin{eqnarray}\label{r-rhlc}
  \| \nabla^2_{TT} f(\bw) - \nabla^2_{TT} f(\bz)  \| \le M_r \| \bw - \bz\|, 
\end{eqnarray}
where $\eta>0$ is a given radius and $T$ can be any index set  {with $|T|= r$} and $T \supseteq\supp(\bx)$.
\end{itemize}
\end{definition}
In the subsequent part of this section, we present  some optimality conditions of (\ref{sco}) which are useful for the algorithmic design. Hereafter, we always let
 \begin{eqnarray}\label{Def-T*}
 \arraycolsep=1.4pt\def\arraystretch{1.5}
\begin{array}{rll}
 {T_* :=\supp (\bx^{*}).}
\end{array}
\end{eqnarray}
 \begin{theorem}\label{theo01-suff-necc} Suppose $f$ is $s$-RC and let $\bx^{*}$ be a point that satisfies
 \begin{eqnarray}\label{l1-KKT}
 \arraycolsep=1.4pt\def\arraystretch{1.5}
\begin{array}{rll}
 \nabla_{T_*} f(\bx^{*}) =0,& &~~{\rm if}~~ \|\bx^{*}\|_{0}=s,\\
{ \nabla f(\bx^{*})  =0},&&~~{\rm if}~~\|\bx^{*}\|_{0}<s.
\end{array}
\end{eqnarray}
 Then $\bx^{*}$ is a local minimizer of  (\ref{sco}) if and only if it satisfies (\ref{l1-KKT}).  Furthermore, the local minimizer   $\bx^{*}$ is unique if   $f$ is locally $s$-RSC at $\bx^*$.
\end{theorem} 
\begin{proof} The first claim follows from \cite[Table 2]{pan2015solutions} and \cite[Theorem 10.1]{RW1998}. We prove the second one. If $\bx^*=0$, then it is unique. Hence we  consider $\bx^*\neq0$. 
Since $\bx^{*}$ be a local minimizer of  (\ref{sco}), there is an $\epsilon>0$ such that
\begin{eqnarray*}
\bx^* = {\rm argmin}~f(\bx),~~{\rm s.t.}~\bx \in S \cap N(\bx^*, \epsilon).
\end{eqnarray*}
Let $\epsilon_*:=\min\{ \epsilon, \min_{i\in T_*}|x_i^*|\}>0$ and then $\bx^*$ is  a global minimizer of the following problem
\begin{eqnarray*}
\bx^* = {\rm argmin}~f(\bx),~~{\rm s.t.}~\bx \in S \cap N(\bx^*, \epsilon_*).
\end{eqnarray*}
The definition of $\epsilon_*$ implies  $T_*:=\supp(\bx^*)\subseteq\supp(\bx)$ for any $\bx \in S \cap N(\bx^*, \epsilon_*)$. If  $\|\bx^{*}\|_{0}<s$, then $ \nabla f(\bx^{*})=0$ by \eqref{l1-KKT}, which by \eqref{r-rsc} with $r=s$ derives that  
\begin{eqnarray*}
f(\bx)    \geq    f(\bx^*) + (\ell_s/2)\| \bx-\bx^*\|^2, 
\end{eqnarray*}
for any $s$-sparse vectors $\bx$. The above condition displays the uniqueness of locally optimal solution $\bx^*$. If $\|\bx^{*}\|_{0}=s$, then  for any $\bx \in S \cap N(\bx^*, \epsilon_*)$, we have $T_*=\supp(\bx)$ due to $|T_*|=\|\bx^{*}\|_{0}=s\geq \|\bx\|_{0}$ and $T_*\subseteq\supp(\bx)$. Using this condition,  \eqref{l1-KKT} and \eqref{r-rsc} with $r=s$ yields
\begin{eqnarray*}  \arraycolsep=1.4pt\def\arraystretch{1.5}
 \begin{array}{llll}
f(\bx) -  f(\bx^*)& \geq & \langle \bx-\bx^*,  \nabla f(\bx^{*})  \rangle +(\ell_s/2)\| \bx-\bx^*\|^2 \\
&= &  \langle (\bx-\bx^*)_{T_*},\nabla_{T_*} f(\bx^{*})   \rangle  +(\ell_s/2)\| \bx-\bx^*\|^2 \\ 
 &=& (\ell_s/2)\| \bx-\bx^*\|^2, 
\end{array}\end{eqnarray*}
which also shows the uniqueness of locally optimal solution $\bx^*$. 
\end{proof} 
 Based on Theorem  \ref{theo01-suff-necc}, however,  conditions \eqref{l1-KKT} mean that there is no useful information on the gradient indexed on $\overline T_*$ when $\|\bx^{*}\|_0=s$. Therefore, we introduce the concept of the $\alpha$-stationary point of   (\ref{sco}). It was first introduced by \cite[Definition 2.3]{Beck13} and known as the $L$-stationary point.
\begin{definition} \label{def3.1}
A point $\bx^{\ast}$ is called an $\alpha$-stationary point  of  (\ref{sco}) if there exists an $\alpha>0$ such that
\begin{eqnarray}\label{alpha-stationary-point}\bx^{*} \in \Pi_{s} (\bx^{*}-\alpha \nabla f(\bx^{*}) ).\end{eqnarray}
\end{definition}
\noindent The $\alpha$-stationary point  can be equivalently expressed as
\begin{eqnarray}\label{l1}
 \arraycolsep=1.4pt\def\arraystretch{1.5}
\begin{array}{lll}
\nabla_{T_*} f(\bx^{*}) = 0,~~\alpha\|\nabla_{\overline T_*} f(\bx^{*})  \|_\infty  \leq  (\bx^*)^\downarrow_{s}.
\end{array}
\end{eqnarray}
In comparison with  \eqref{l1-KKT},  condition \eqref{l1}  provides more information on the gradient indexed on $\overline{T}_*$. It can be clearly seen that the latter suffices to the former and thus is a stronger condition.  The following results reveal the relationships between an $\alpha$-stationary point and a local/global minimizer  of   (\ref{sco}).

\begin{theorem}\label{theo01} The following statements are true for  (\ref{sco}).
\begin{itemize}
\item[i)] An $\alpha$-stationary point is a local minimizer if $f$ is $s$-RC.
\item[ii)] A global minimizer is an $\alpha$-stationary point with $0<\alpha\leq 1/L_{2s}$ if $f$ is $2s$-RSS.
\item[iii)] Let $\bx^{*}$ be an $\alpha$-stationary point and suppose $f$ is $2s$-RSC with a constant $\ell_{2s}>0$. If $\|\bx^{*}\|_{0}<s$, then it is a unique global minimizer. If $\|\bx^{*}\|_{0}=s$, then it is a unique global minimizer if  $\alpha > 1/\ell_{2s}$.
\end{itemize}   
\end{theorem}
\begin{proof}  The conclusions in i) and ii) can be made by Theorem \ref{theo01-suff-necc} and \cite[Theorem 2.2]{Beck13}, respectively. We prove iii).   If $\|\bx^{*}\|_{0}<s$, then $(\bx^*)^\downarrow_{s} =0 $ leading to $ \nabla f(\bx^{*})=0$ from \eqref{l1}, which by \eqref{r-rsc} derives that, for any $s$-sparse vectors $\bx$, 
\begin{eqnarray}\label{global<s}
f(\bx)   \geq   f(\bx^{*}) + (\ell_{2s}/2) \|\bx-\bx^*\|^2. 
\end{eqnarray}
This shows that $\bx^*$  is   a unique global minimizer. If $\|\bx^{*}\|_{0}=s$, then the definition of $\Pi_s(\cdot)$ in \eqref{alpha-stationary-point} indicates that
$$\|\bx^*-(\bx^{*}-\alpha \nabla f(\bx^{*}))\|\leq\|\bx -(\bx^{*}-\alpha \nabla f(\bx^{*}))\|$$
for any $s$-sparse vectors $\bx$, which results in
$$ 2\alpha\langle f(\bx^{*}), \bx- \bx^{*} \rangle \geq - \|\bx- \bx^{*} \|^2.$$
Using this fact and \eqref{r-rsc} yields that, for any $s$-sparse vectors $\bx$, 
\begin{eqnarray}
\arraycolsep=1.4pt\def\arraystretch{1.5}
 \begin{array}{llll}
2f(\bx)   &\geq&    2f(\bx^{*}) + 2\langle \nabla f(\bx^{*}), \bx- \bx^{*} \rangle+  \ell_{2s}  \|\bx-\bx^*\|^2 \\
&\geq&    2f(\bx^{*}) +   (\ell_{2s} -1/ \alpha ) \|\bx-\bx^*\|^2,
\end{array}
\end{eqnarray}
which also displays the unique global optimality of $\bx^*$ since $\alpha>1/\ell_{2s}$.   
\end{proof}

\section{Gradient Projection Newton Pursuit} \label{sec:htnp}
Before the main results ahead of us, for a computed $s$-sparse point $\bu^k$, we choose an index set $\Gamma_k$ satisfying
  \begin{eqnarray}\label{def-gamma-k}\Gamma_k\supseteq{\rm supp}(\bu^k),\qquad|\Gamma_k|=s,  \end{eqnarray}
and denote
$$ \bg^k:= \nabla f(\bu^k),\qquad\H^k:= \nabla^2_{\Gamma_k\Gamma_k} f(\bu^k).$$
We point out that if $\|\bu^k\|_0=s$, then $\Gamma_k$ is unique, namely, $\Gamma_k = {\rm supp}(\bu^k)$. If $\|\bu^k\|_0<s$, then there are multiple choices for $\Gamma_k$. We just pick one of them. The algorithmic framework of  gradient projection Newton pursuit (GPNP)   is presented in Algorithm \ref{algorithm 1} and consists of two major components: 
\begin{algorithm}[!th]
\SetAlgoLined
 Initialize $\bx^{0}$,  $\tau >0, \sigma >0, 1>\gamma>0, \pi_0 >\varepsilon>0, \epsilon>0$ and {set $k=0$}.\\
 \While{$\pi_k > \varepsilon$}{
  \underline{Gradient projection:} Find the smallest integer $q_k = 0, 1, \ldots $ such that \\
 \parbox{.92\textwidth}{\begin{eqnarray}\label{armijio-descent-property} 
 f(\bx^k(\tau\gamma^{q_{k}}))\leq f(\bx^k)-(\sigma/2)\|\bx^k(\tau\gamma^{q_{k}})-\bx^{k}\|^{2},
\end{eqnarray}}\\
{ where $\bx^k(\alpha)$ is defined by (\ref{notation-k-1}).}  Set $\alpha_k= \tau\gamma^{q_{k}}$,  $ \bu^k=\bx^k(\alpha_k)$ and $\bx^{k+1}=\bu^{k}$.\\
 \underline{Newton pursuit:}  \If{$\supp(\bx^k)=\Gamma_k$ or $\| \bg^k\|<\epsilon$}{
 If the following equations are solvable\\
   \parbox{.88\textwidth}{\begin{eqnarray}\label{Newton-descent-property}
	  \H^k (\bv^{k}_{\Gamma_k}-\bu^{k}_{\Gamma_k}) =   -  \bg_{\Gamma_k}, ~~~~\bv^{k}_{\overline\Gamma_k}=0,
\end{eqnarray}}\\
 and the solution $\bv^k$ satisfies\\
    \parbox{.88\textwidth}{\begin{eqnarray}
    \label{Newton-descent-property-1}
	 f(\bv^{k})  \leq f(\bu^{k})-(\sigma/2)  \|\bv^k-\bu^k\|^2 ,
\end{eqnarray}}\\
then set $\bx^{k+1}=\bv^{k}$.
   }
Compute $\pi_{k+1}$ and set $k= k+1$.
 }
 Output the solution $\bx^k.$
 \caption{GPNP: Gradient Projection Newton Pursuit \label{algorithm 1}}
\end{algorithm}
\begin{itemize}
\item[I)] {The main steps select a point, $\bx^k(\alpha)$, from the hard-thresholding operator/gradient projection, namely,
   \begin{eqnarray}\label{notation-k-1}
  \bx^k(\alpha) \in  \Pi_{s}(\bx^{k}-\alpha \nabla f(\bx^k)),
   \end{eqnarray}
   where $\alpha>0$ is the step size that is chosen properly so as to 
 make the objective function value decreasing with a desirable scale. We emphasize that the right-hand side of \eqref{notation-k-1} is a set relying on $\alpha$ for given $\bx^{k}$. Any point in the set can be used to define point $\bx^k(\alpha)$. It is easy to see that if $\alpha=0$ and  $\|\bx^{k}\|_0\leq s$, then $\bx^k(0) = \bx^k$ due to $\Pi_{s}(\bx^{k})=\{\bx^{k}\}$.}
 
  \item[II)] The second part adopts a Newton step to speed up the convergence. However, Newton steps are only performed when one of the following conditions is satisfied, 
\begin{eqnarray} \label{Newton-switch-on}
\supp(\bx^k) = \Gamma_k,~~~~\| \bg^k\|<\epsilon, 
\end{eqnarray}
where  $\epsilon>0$ is a given tolerance. {The above two conditions can be deemed as the conditions for checking the neighbourhood of a locally (or globally) optimal solution, say ${\bf x}^*$. If $\|{\bf x}^*\|_0=s$, then points around it have the same support, corresponding to the first condition. If $\|{\bf x}^*\|_0<s$, then points around it have a small gradient due to \eqref{l1-KKT}, corresponding to the second condition. Therefore, once a point ${\bf x}^k$ meets one of these conditions, it falls into the neighbourhood of a locally (or globally) optimal solution. Then we perform Newton steps to find it quickly. In a nutshell, using \eqref{Newton-switch-on} as a switch of Newton steps enables to accelerate the convergence. }
\end{itemize}

\begin{remark}For the halting condition of GPNP  in Algorithm \ref{algorithm 1}, in our numerical experiments, we set $\pi_k$ by
\begin{eqnarray} \label{tol}
	 \pi_k :=
	  \arraycolsep=1.4pt\def\arraystretch{1.5}
 \left\{ \begin{array}{lll}
    \|  \nabla f(\bx^k)\| , & k<k_0,\\
  \max \left\{{\rm std}\left(f_k,f_{k-1},\ldots ,f_{k-k_0}\right),\|  \nabla f(\bx^k)\|\right \},& k \geq k_0, 
 \end{array}
 \right.
\end{eqnarray}
where $f_k:=f(\bx^k)$, $k_0$ is an integer (e.g. $k_0=5$ in the numerical experiments) and ${\rm std}(\bx)$ calculates the standard deviation of $\bx$. In the next section,  Theorem \ref{global-convergence} will show that whole sequence $\{\bx^k\}$  converges to an $\alpha$-stationary point  (say $\bx^*$) satisfying $\|\bx^*\|_0\leq s$. As a result,
\begin{eqnarray}  \label{two-cases}
 \arraycolsep=1.4pt\def\arraystretch{1.5}
 \begin{array}{lll}
\text{if}~~\|\bx^*\|_0=s~~\text{then}~~\supp(\bx^k)=\Gamma_k,\\
\text{if}~~\|\bx^*\|_0<s~~\text{then}~~\|  \bg^k\|<\epsilon,
\end{array}
\end{eqnarray} 
 for sufficiently large $k$.   The above relationships (see the proof of Theorem \ref{quadratic-lemma}) imply that the conditions for Newton steps will eventually be satisfied.  What is more,   the first relation in (\ref{two-cases}) enables to guarantee sufficiently small ${\rm std}(f_k,f_{k-1},\ldots ,f_{k-k_0})$, while the second relation can ensure a small value of $\| \bg^k\|$ due to $\bg^k \rightarrow \nabla f(\bx^*)=0$ (see the proof of Theorem \ref{quadratic-lemma}). Therefore, the designed halting condition (\ref{tol}) makes sense.   
 
\end{remark}

 \begin{remark}\label{remark-nhtp-gpnp}
 We note that NHTP proposed in \cite{zhou2021global} also integrates the Hard-thresholding operator and  Newton steps. We would like to emphasize that the algorithmic frameworks of NHTP and GPNP are different. The former takes Newton steps as its main steps whilst using gradient descent as compensation only when Newton directions violate certain conditions. By contrast, gradient projections are dominant steps in GPNP, and Newton directions are performed only when certain conditions are met.
\end{remark}
	
\subsection{Global convergence}
In this subsection, we aim to establish the main convergence results. The first result shows that step size $\alpha_k$ is bounded away from zero and the sequence of objective function $\{f(\bx^{k})\}$ is strictly decreasing.
\begin{lemma} \label{theo05} Suppose $f$ is  $2s$-RSS with $L_{2s}>0$. The following results hold for sequence $\{\bx^{k}\}$ generated by GPNP.
\begin{itemize}
\item[1)] For any $0<\alpha<1/(\sigma+L_{2s})$, it holds that
		\begin{eqnarray}
		\label{zk-alpha-zk}~~f(\bx^k(\alpha))\leq f(\bx^{k})- (\sigma/2) \|\bx^k(\alpha) -\bx^{k}\|^{2},
		\end{eqnarray}
which hence leads to
		\begin{eqnarray}
		\label{alpha-k} \inf_{k\geq 0} \alpha_k \geq \underline\alpha :=\min\left\{\tau, {\gamma} /{(\sigma+ L_{2s} )}  \right\}>0.
		\end{eqnarray}	
\item[2)] $\{f(\bx^k)\}$ is a strictly decreasing sequence and $$\lim\limits_{k\rightarrow \infty}\|\bu^{k}-\bx^{k}\|=\lim\limits_{k\rightarrow \infty}\|\bx^{k+1}-\bx^{k}\|=0.$$ 
\item[3)] Any accumulating point of $\{\bx^k\}$ is an $\alpha$-stationary point  of (\ref{sco}) with any $0<\alpha\leq\underline{\alpha}$. 
\end{itemize}
\end{lemma}
\begin{proof}1)   It follows from \eqref{notation-k-1} that  
$$\| \bx^k(\alpha) -( \bx ^k - \alpha  \nabla f(\bx^k))\|^2 \leq \|\bx^k -( \bx^k - \alpha \nabla f(\bx^k))\|^2,$$
 which results in
   \begin{eqnarray*} \label{eta-point-not-3}
 2\alpha\langle   \nabla f(\bx^k),  \bx^k(\alpha)-  \bx^k\rangle \leq -  \| \bx^k(\alpha)-  \bx^k\|^2.
  \end{eqnarray*}
Using the above condition and $f$ being the $2s$-RSS   derives that
\begin{eqnarray*}
 \arraycolsep=1.4pt\def\arraystretch{1.5}
\begin{array}{llll}
2f(\bx^k(\alpha)) &\leq&    2f(\bx^k)+
  2 \langle  \nabla f(\bx^k),\bx^k(\alpha)-\bx^k\rangle
+  L_{2s}  \| \bx^k(\alpha)-\bx^k \|^2\\
& {\leq}& 2f(\bx^k) - ( {1}/{ \alpha } - L_{2s} )\|\bx^k(\alpha)-\bx^k\|^2\\
&\leq& 2f(\bx^k)  -   \sigma \|\bx^k(\alpha)-\bx^k\|^2. ~~~(\text{by~$0<\alpha\leq  1/(\sigma+L_{2s})$})
\end{array}
\end{eqnarray*}
The above relationship indicates that \eqref{armijio-descent-property} can be met as long as $\alpha_k=\tau\gamma^{q_k}\leq 1/(\sigma+L_{2s})$, thereby resulting in  $\alpha_k\geq \tau\gamma^{q_k+1}\geq \gamma/(\sigma+L_{2s})$. This and $\alpha_k\leq \tau$ prove the desired result.

2)  By the framework of Algorithm \ref{algorithm 1} that  $ \bu^k=\bx^k(\alpha_k)$ and  \eqref{zk-alpha-zk}, it follows
\begin{eqnarray}
\label{fact-u-z-0}		2f(\bu^{k}) \leq 
		 2f(\bx^k)-   \sigma  \|\bu^{k}-\bx^k\|^2.
		\end{eqnarray}
If $\bx^{k+1}=\bu^k$, then  we obtain
\begin{eqnarray*}
\label{fact-u-z}		2f(\bx^{k+1}) \leq 
		 2f(\bx^k)-   \sigma  \|\bx^{k+1}-\bx^k\|^2.
		\end{eqnarray*}
If $\bx^{k+1}=\bv^k$, then  we obtain
\begin{eqnarray*}\label{fact-u-z-1}
 \arraycolsep=1.4pt\def\arraystretch{1.5}
\begin{array}{lllr}
2f(\bx^{k+1})= 2f(\bv^{k}) 
 &\leq& 2f(\bu^{k}) -  \sigma \|\bv^{k} -\bu^k\|^2 &~~(\text{by \eqref{Newton-descent-property-1}})\\
&=&  2f(\bu^{k}) -  \sigma \|\bx^{k+1} -\bu^k\|^2 \\
&\leq&   2f(\bx^{k}) -   \sigma \|\bu^k-\bx^k\|^2 -  \sigma \|\bx^{k+1}-\bu^k\|^2&~~(\text{by \eqref{zk-alpha-zk}})\\
&\leq&   2f(\bx^{k})  -(\sigma/2)\|\bx^{k+1}-\bx^k\|^2, 
\end{array}
\end{eqnarray*}
where the last inequality is from fact $\|{\bf a}+{\bf b}\|^2\leq 2\|{\bf a}\|^2+2\|{\bf b}\|^2$  {for all vectors ${\bf a}$ and ${\bf b}$}. Both cases lead to	 \begin{eqnarray}\label{decrese-pro}
 \arraycolsep=1.4pt\def\arraystretch{1.5}
\begin{array}{lll}
		 2f(\bx^{k+1}) &\leq& 2f(\bx^{k})  - (\sigma/2)\|\bx^{k+1}-\bx^k\|^2,\\  
		  2f(\bx^{k+1}) &\leq& 2f(\bx^{k}) -   \sigma \|\bu^k-\bx^k\|^2. 
\end{array}
\end{eqnarray}
Therefore, $\{f(\bx^{k})\}$ is a non-increasing sequence, resulting in
		\begin{eqnarray*}
		 \arraycolsep=1.4pt\def\arraystretch{1.5}
\begin{array}{lll}
 \sum_{k\geq 0} \max\left\{\frac{\sigma}{4}\|\bx^{k+1}-\bx^k\|^2,  \frac{\sigma}{2} \|\bu^k-\bx^k\|^2\right\} 
		&\leq& \sum_{k\geq 0} \left[ f(\bx^{k}) - f(\bx^{k+1}) \right]~~~~(\text{by \eqref{decrese-pro}})\\
		&= &f(\bx^{0}) - \lim_{k\rightarrow\infty}f(\bx^{k})\\
		&\leq& +\infty.~~~~(\text{by $f$ being bounded from below})\\
		\end{array}	
		\end{eqnarray*}
The above condition suffices to $\lim_{k\rightarrow\infty}\|\bx^{k+1}-\bx^k\|=\lim_{k\rightarrow\infty}\|\bu^k-\bx^k\|=0.$

3) Let $\bx^*$ be any accumulating point of $\{\bx^k\}$. Conclusion 2) indicates that there exists a subset $\Omega$ of $\{0,1,2,\ldots \}$ such that both $\{\bx^k:k\in \Omega\}$ and $\{\bu^k:k\in \Omega\}$  converge to $\bx^*$. Moreover, by 1) that $\{\alpha_k\}$  belongs to a bounded interval $[\underline{\alpha} ,\tau]$, there is subsequence $K$ of $\Omega$ such that $\{\alpha_k:k\in K\}$ converges to an accumulating point (say $\alpha_*$).  So, we have 
\begin{eqnarray}
		\label{a-a-under} \lim_{ k (\in K)\rightarrow\infty}\bx^k =\lim_{ k (\in K)\rightarrow\infty}\bu^k = \bx^*,~~~~\lim_{ k (\in K)\rightarrow\infty}\alpha_k = \alpha_*\in[ \underline{\alpha} ,\tau].\end{eqnarray}
In the sequel, we prove $\bx^*$ is an $ \alpha$-stationary point. To proceed with that, let $\be^k :=\bx^k -\alpha_k \nabla f(\bx^k)$. The framework of Algorithm \ref{algorithm 1} implies
\begin{eqnarray}\label{u-l-P}
\bu^k  \in \Pi_{s}( \be^k ),~~~~ \lim_{ k (\in K)\rightarrow\infty} \be^k =\bx^* -  \alpha_*  \nabla f(\bx^*)=:\be^*.\end{eqnarray}
The first condition means $\bu^k  \in   {S}$ for any $ k \geq1$. Note that $ {S}$ is closed and  $\bx^*$ is the accumulating point of $\{\bu^k\}$ by \eqref{a-a-under}. Therefore, $\bx^*\in  {S}$, which results in
\begin{eqnarray}
 \label{z-*-g}
\min_{\bx\in  {S}}\|\bx -\be^*\| \leq \|\bx^*-\be^*\|.\end{eqnarray}
If the strict inequality holds in the above condition, then there is an $ \varepsilon_0>0$ such that
\begin{eqnarray*}
\|\bx^*-\be^*\|-\varepsilon_0 &=&  \min_{\bx\in {S}}\|\bx -\be^*\| \\
&\geq&\min_{\bx\in  {S}} \left(\|\bx - \be^k \|-\|\be^k- \be^*\|\right)\\
&=&\|\bu^k - \be^k \|-\|\be^k -\be^*\|. ~~~(\text{by~ \eqref{u-l-P}})
\end{eqnarray*}
Taking the limit of both sides of the above condition along $ k (\in K)\rightarrow\infty$ yields $\|\bx^*- \be^*\|-\varepsilon_0 \geq \|\bx^* - \be^*\|$ by \eqref{a-a-under} and \eqref{u-l-P}, a contradiction with  $ \varepsilon_0>0$. Therefore, we must have the equality in \eqref{z-*-g}, showing that
\begin{eqnarray*} \bx^* \in  \Pi_{s}(\be^*)= \Pi_{s}\left( \bx^* -   \alpha_*     \nabla f(\bx^*) \right). \end{eqnarray*}
Therefore, $\bx^*$ is an $ \alpha$-stationary point of \eqref{sco} with any $\alpha\in(0,\underline{\alpha}]$ due to $\underline{\alpha}\leq  \alpha_*$.
\end{proof}
To claim the global convergence of the whole sequence, we first need to show the boundedness of the sequence and then prove that any of its accumulating points is isolated. Finally, combining the above lemma and  \cite[Lemma 4.10]{more1983computing} can conclude  the desired result.
 \begin{theorem}\label{global-convergence} If $f$ is  $2s$-RSS  and $s$-RSC, then the whole sequence converges to a unique local minimizer of (\ref{sco}).
\end{theorem}
\begin{proof} The boundedness of sequence $\{\bx^k\}$ follows from three facts: $f$ being $s$-RSC, $\|\bx^k\|_0\leq s$, and $f(\bx^k)\leq f(\bx^0)$ by \eqref{zk-alpha-zk}. The boundedness ensures the existence of accumulating points of sequence $\{\bx^k\}$.  We note that any accumulating point (say $\bx^*$) is an $ \alpha$-stationary point (see Lemma \ref{theo05}) which  is also a local minimizer (see Theorem \ref{theo01}). This and Theorem \ref{theo01-suff-necc} indicate that $\bx^*$ is unique if $f$ is  (locally) $s$-RSC at $\bx^*$. In other words, $\bx^*$ is an isolated local minimizer of (\ref{sco}).  Finally, using three facts: $\bx^*$ being isolated, \cite[Lemma 4.10]{more1983computing}, and $\|\bx^{ k +1}-\bx^ k \|\rightarrow0$ by Lemma \ref{theo05}, we can conclude that
the whole sequence converges to  $\bx^*$, the unique local minimizer.  
\end{proof}
\subsection{Locally quadratic convergence}
Besides having the global convergence property, the support set of the sequence generated by GPNP can be identified within finitely many steps. Then Newton steps can always be admitted eventually, resulting in a quadratic convergence rate. 
\begin{theorem}\label{quadratic-lemma} Suppose $f$ is  $2s$-RSS  with $L_{2s}>0$ and  $s$-RSC with $\ell_s>0$. Let  $\{\bx^ k \}$ be the sequence generated by GPNP  and $\bx^*$ be its limit.  Then   
\begin{itemize}
\item[1)] for sufficiently large $k$, the support set of $\bx^*$ can be identified by
\begin{eqnarray}\label{support-identify}
 \arraycolsep=1.4pt\def\arraystretch{1.5}
 \begin{array}{lll}
\supp(\bx^*)\subseteq (\supp(\bx^k) \cap \Gamma_k),&~~\text{if}~~ & \|\bx^*\|_0<s,\\
\supp(\bx^*)\equiv \supp(\bx^k)\equiv \Gamma_k,&~~\text{if}~~& \|\bx^*\|_0=s.
\end{array} 
\end{eqnarray}
\end{itemize}
If we further assume $f$ is locally $s$-RHLC at $ \bx^*$  with $M_s^*>0$ and set $\sigma\in(0,\ell_s/4]$, then
\begin{itemize}
\item[2)] Newton steps are always admitted for sufficiently large $k$,
\item[3)] and the sequence eventually converges to $\bx^*$ quadratically, namely,
\begin{eqnarray} \label{quadratic-rate-formula}
 \arraycolsep=1.4pt\def\arraystretch{1.5}
 \begin{array}{lll}
\|\bx^{k+1}-\bx^{*}\| \leq  \frac{M_s^*(1+ \tau L_{2s})^2 }{\ell_s} \|\bx^{k}-\bx^*\|^{2}, \qquad \text{for sufficiently large $k$.}
\end{array} 
\end{eqnarray}  
\end{itemize}
\end{theorem}
\begin{proof}1) By Theorem \ref{global-convergence}, the whole sequence converges to $\bx^*$. For  case  $\|\bx^*\|_0=s$, suppose $\supp(\bx^*) 
\neq \supp(\bx^k) $ for sufficiently large $k$, then there is an $i$ such that $i\in \supp(\bx^*)$ but  $i\notin \supp(\bx^k)$ due to $\|\bx^k\|_0\leq s$ and $\|\bx^*\|_0=s$. This causes 
$$0< | x_i^*| \leq \|\bx^k-\bx^*\|\rightarrow 0,$$
a contradiction, thereby proving $\supp(\bx^*) 
\equiv \supp(\bx^k) $. Similarly, we can show  $\supp(\bx^*) 
\equiv \supp(\bu^k) \subseteq \Gamma_k $ by \eqref{def-gamma-k}, leading to $\supp(\bx^*)   = \Gamma_k $ due to $|\supp(\bx^*) |=|\Gamma_k|=s$. For the case of $\|\bx^*\|_0<s$,  similar reasoning allows for deriving $\supp(\bx^*) 
\subseteq   \supp(\bx^k) $ and $ \supp(\bx^*) 
\subseteq \supp(\bu^k) \subseteq \Gamma_k$.

2)  The following statements are given for sufficiently large $k$ if no additional explanations are provided. Before we show the quadratic convergence, we need several facts. The updating rule, \eqref{Newton-descent-property}, for $\bv^k$ indicates that $\supp(\bv^k)\subseteq\Gamma_k$. This together with $\supp(\bu^k)\subseteq\Gamma_k$ and \eqref{support-identify} ensures the following observations,
\begin{eqnarray} \label{Gv-Gu} \arraycolsep=1.4pt\def\arraystretch{1.5}
 \begin{array}{lll}
	 \| \bu^k-\bx^* \|=\|(\bu^k-\bx^*)_{\Gamma_k}\|,\qquad
	  \| \bu^k-\bv^k \|=\|(\bu^k-\bv^k)_{\Gamma_k}\|.
	  \end{array}
\end{eqnarray}
It follows conditions \eqref{l1-KKT} and \eqref{support-identify} that
\begin{eqnarray}\label{b*-gamma}
\nabla_{\Gamma_k}f(\bx^*)=0.
\end{eqnarray}	
The assumptions that $f$ is $2s$-RSS with $L_{2s}$ and  $s$-RSC with $\ell_s$ give us
\begin{eqnarray}\label{H-bounded}
 \arraycolsep=1.4pt\def\arraystretch{1.5}
 \begin{array}{llll}
\lambda_{\max}( \nabla^2_{T,T} f(\bx)) &\leq& L_{2s}, &\forall~ |T|=s, T\supseteq\supp(\bx), \forall~ \bx\in S,\\
 \lambda_{\min}( \nabla^2_{T,T} f(\bx^*)) &\geq& \ell_s, &\forall ~|T|=s, T\supseteq\supp(\bx^*). 
 \end{array}
\end{eqnarray}
where $\lambda_{\max}(H)$ and $\lambda_{\min}(H)$ stand for the largest and smallest eigenvalues of a symmetric matrix $H$. {Recalling the definition of  $s$-RHLC in  \eqref{r-rhlc}, we take $\eta:=\ell_s/(2M_s^*)$. For sufficiently large $k$, $\|\bu^k-\bx^*\|\leq\eta$ due to $\bu^k\to\bx^*$ and hence
\begin{eqnarray}\label{H-u-bounded}
 \arraycolsep=1.4pt\def\arraystretch{1.5}
 \begin{array}{llll}
 \lambda_{\min}( \H^k) &\geq& \lambda_{\min}(\nabla^2_{\Gamma_k,\Gamma_k} f(\bx^*)) - \| \H^k - \nabla^2_{\Gamma_k,\Gamma_k} f(\bx^*)\|\\
 &\geq& \ell_s- M^*_s\|\bu^k-\bx^*\|  \geq \ell_s/2.\qquad (\text{by $s$-RHLC and \eqref{H-bounded}})\\
  \end{array}
\end{eqnarray}}
 By letting  $\bu^k_t:=t\bu^k+(1-t)\bu^*$ for $t\in[0,1]$, it follows $\supp(\bu^k_t) \subseteq \Gamma_k$ and thus 
\begin{eqnarray}\label{v-u-00}
 \arraycolsep=1.4pt\def\arraystretch{1.5}
 \begin{array}{lllr}
\|\nabla_{\Gamma_k}f(\bu^k)-\nabla_{\Gamma_k}f(\bx^*)\|
  &=& \|\int_0^1\nabla^2_{\Gamma_k:}f(\bu^k_t)(\bu^k-\bx^*)dt\|&(\text{by Mean Value Theory}) \\
   &=& \|\int_0^1\nabla^2_{\Gamma_k\Gamma_k}f(\bu^k_t)(\bu^k-\bx^*)_{\Gamma_k} dt\|&(\text{by $\supp(\bx^*) \subseteq \Gamma_k$}) \\
 &\leq& L_{2s}\|(\bu^k-\bx^*)_{\Gamma_k}\|&(\text{by $\supp(\bu^k_t) \subseteq \Gamma_k$ and \eqref{H-bounded}})\\
 &=& L_{2s}\|\bu^k-\bx^* \|.&(\text{by \eqref{Gv-Gu}}) 
 \end{array}
\end{eqnarray}

Now, we are ready to conclude the concision.  Firstly, the two relations in \eqref{two-cases} are true. In fact, the limit satisfies $\|\bx^*\|_0 \leq s$. If $\|\bx^*\|_0 =s$, then the first condition  in \eqref{two-cases} can be guaranteed by \eqref{support-identify} immediately. If $\|\bx^*\|_0<s$, we have $\nabla f(\bx^*)=0$ by \eqref{l1-KKT}, which together with $ \bu^k \rightarrow \bx^*$ leads to $\bg^k=\nabla f(\bu^k)\rightarrow 0$, thereby ensuring the second condition  in \eqref{two-cases}. Therefore,  conditions \eqref{Newton-switch-on} for switching on Newton steps are satisfied for sufficiently large $k$.  

Secondly, the following fact,
\begin{eqnarray}\label{v-u-0}
 \arraycolsep=1.4pt\def\arraystretch{1.5}
 \begin{array}{lllr}
  \|\bv^{k} -\bu^{k}\|&=&\|(\bv^{k} -\bu^{k})_{\Gamma_k}\|&(\text{by \eqref{Gv-Gu}})\\ 
  &=&\|( \H^k)^{-1} \nabla_{\Gamma_k}f(\bu^k)\|&(\text{by  \eqref{Newton-descent-property}})\\ 
   &=&\|( \H^k)^{-1} (\nabla_{\Gamma_k}f(\bu^k)-\nabla_{\Gamma_k}f(\bx^*))\| &(\text{by \eqref{b*-gamma}}) \\
 &\leq& (2/\ell_s)\|\nabla_{\Gamma_k}f(\bu^k)-\nabla_{\Gamma_k}f(\bx^*)\|&(\text{by \eqref{H-u-bounded}}) \\
 &\leq& (2L_{2s}/\ell_s)\| \bu^k-\bx^*\|&(\text{by \eqref{v-u-00}}) \\
 & \rightarrow& 0,&(\text{by $ \bu^k \rightarrow \bx^*$  }) 
 \end{array}
\end{eqnarray}
allows us to deliver the chain of inequalities,
\begin{eqnarray}\label{v-u-0-0}
 \arraycolsep=1.4pt\def\arraystretch{1.5}
\begin{array}{lllr}
2f(\bv^{k})-2f(\bu^{k})
&=& 2\langle  \bg^{k}, \bv^{k} -\bu^{k}\rangle+  \langle \nabla^2f(\bu^{k}) (\bv^{k} -\bu^{k}), \bv^{k} -\bu^{k} \rangle + o(\|\bv^{k} -\bu^{k}\|^2)\\
&=& 2\langle \bg^{k}_{\Gamma_k}, (\bv^{k} -\bu^{k})_{\Gamma_k} \rangle+  \langle \H^{k} (\bv^{k} -\bu^{k})_{\Gamma_k}, (\bv^{k} -\bu^{k})_{\Gamma_k} \rangle + o(\|\bv^{k} -\bu^{k}\|^2)&(\text{by \eqref{Gv-Gu}})\\
&=& -\langle \H^{k} (\bv^{k} -\bu^{k})_{\Gamma_k}, (\bv^{k} -\bu^{k})_{\Gamma_k} \rangle + o(\|\bv^{k} -\bu^{k}\|^2)&(\text{by \eqref{Newton-descent-property}})\\
&\leq&  - (\ell_s/2) \|(\bv^{k} -\bu^{k})_{\Gamma_k}\|^2 + o(\|\bv^{k} -\bu^{k}\|^2)  &(\text{by \eqref{H-u-bounded}}) \\
&\leq&  - (\ell_s/4) \| \bv^{k} -\bu^{k}\|^2 ~&(\text{by \eqref{Gv-Gu}, \eqref{v-u-0}})\\
&\leq& -\sigma \|\bv^{k} -\bu^{k}\|^2. &(\text{by   $0<\sigma\leq\ell_s/4$})
\end{array}
\end{eqnarray}
Overall, Newton steps always are admitted  for sufficiently large $k$. 
 
3)  Recall \eqref{notation-k-1} that $\bu^k  \in  \Pi_{s}(\bx^{k}- \alpha_k \nabla f(\bx^k) ) 
$ and $\supp(\bu^k)\subseteq\Gamma_k$. If $\|\bu^k\|_0=s$, then $\supp(\bu^k)=\Gamma_k$ due to $|\supp(\bu^k)|=|\Gamma_k|=s$, resulting in following relationship 
\begin{eqnarray} \label{uk-xk-Gammak}
	 \arraycolsep=1.4pt\def\arraystretch{1.5}
\begin{array}{lllr}\bu^k_{\Gamma_k}=\bx^{k}_{\Gamma_{k}}-\alpha_k \nabla_{\Gamma_k} f(\bx^k),\qquad \bu^k_{\overline\Gamma_k}=0\end{array}
\end{eqnarray}
If  $\|\bu^k\|_0<s$, condition $\bu^k  \in  \Pi_{s}(\bx^{k}- \alpha_k \nabla f(\bx^k) )$ indicates $\|\bx^{k}- \alpha_k \nabla f(\bx^k)\|_0<s$ and hence $\bu^k=\bx^{k}- \alpha_k \nabla f(\bx^k)$. Then the above relationship is also valid.  Consequently, it follows
\begin{eqnarray} \label{uk-x*}
	 \arraycolsep=1.4pt\def\arraystretch{1.5}
\begin{array}{lllr}
\|\bu^{k}-\bx^*\|  
&=& \|\bu^{k}_{\Gamma_{k}} - \bx^*_{\Gamma_{k}}\| 
= \|\bx^{k}_{\Gamma_{k}}-\alpha_k \nabla_{\Gamma_k} f(\bx^k) - \bx^*_{\Gamma_{k}}\| &(\text{by \eqref{Gv-Gu}})\\
 &=& \|\bx^{k}_{\Gamma_{k}}-\alpha_k \nabla_{\Gamma_k} f(\bx^k) - \bx^*_{\Gamma_{k}}-\alpha_k \nabla_{\Gamma_k} f(\bx^{*}) \| &(\text{by   \eqref{b*-gamma}})\\
&\leq& \|\bx^{k}_{\Gamma_{k}}- \bx^*_{\Gamma_{k}}\|+\alpha_k \|\nabla_{\Gamma_k} f(\bx^k) -  \nabla_{\Gamma_k} f(\bx^{*}) \|\\
&\leq& \|\bx^{k} - \bx^* \|+    \tau  L_{2s} \| \bx^{k} - \bx^* \| &(\text{by  $0< \alpha_k\leq  \tau $,  \eqref{v-u-00}})\\
&=& (1+  \tau  L_{2s})\|\bx^{k} - \bx^* \|.  
\end{array}
\end{eqnarray}
Finally, since Newton steps always are taken as $\bx^{k+1}=\bv^{k}$  for sufficiently large $k$, we obtain 
\begin{eqnarray*} 
	 \arraycolsep=1.4pt\def\arraystretch{1.5}
\begin{array}{lllr}
\|\bx^{k+1}-\bx^*\| &=& \|\bv^{k}-\bx^*\| = \|\bv^{k}_{\Gamma_k}-\bx^*_{\Gamma_k}\| &(\text{by \eqref{Gv-Gu}})\\
&=& \|\bv^{k}_{\Gamma_k}-\bu^{k}_{\Gamma_k}+\bu^{k}_{\Gamma_k}-\bx^*_{\Gamma_k}\|  \\
&=& \|( \H^k)^{-1} (\nabla_{\Gamma_k}f(\bu^k)-\nabla_{\Gamma_k}f(\bx^*))+\bu^{k}_{\Gamma_k}-\bx^*_{\Gamma_k}\|  &(\text{by  \eqref{v-u-0}}) \\
& \leq& (2/\ell_s) \| \nabla_{\Gamma_k}f(\bu^k)-\nabla_{\Gamma_k}f(\bx^*) +\H^k(\bu^{k}_{\Gamma_k}-\bx^*_{\Gamma_k})\| &(\text{by  \eqref{H-u-bounded}}) \\
& =& (2/\ell_s) \| \int_0^1(\nabla^2_{\Gamma_k \Gamma_k}f(\bu^k_t) - \H^k)(\bu^{k}_{\Gamma_k}-\bx^*_{\Gamma_k}) dt\| &(\text{by \eqref{v-u-00}})    \\

& \leq & (2/\ell_s)   \int_0^1\|\nabla^2_{\Gamma_k \Gamma_k}f(\bu^k_t) - \H^k\|\cdot\|\bu^{k}_{\Gamma_k}-\bx^*_{\Gamma_k}\| dt     \\
& \leq & (2/\ell_s)  \int_0^1 M^*_s\|\bu^k_t-\bu^k\|\cdot\|\bu^{k}_{\Gamma_k}-\bx^*_{\Gamma_k}\| dt   &(\text{by   $s$-RHLC})  \\
& = & (2M^*_ s/\ell_s) \| \bu^k -\bx^*  \|\cdot\|\bu^{k}_{\Gamma_k}-\bx^*_{\Gamma_k}\| \int_0^1(1-t) dt    \\
& = & (M^*_ s/\ell_s) \| \bu^k -\bx^*\|^2,&(\text{by \eqref{Gv-Gu}}) 
\end{array}
\end{eqnarray*}  
which together with \eqref{uk-x*} draws the conclusion immediately. 
\end{proof}
 
\begin{remark} It is worth pointing out that if $\|\bx^*\|_0=s$, then the support of $\bx^k$ will be identified eventually, namely, $\bx^k\in \Omega:=\{\bx\in\R^n: \supp(\bx)=\supp(\bx^*)\}.$ In this regard, the quadratic convergence in Theorem \ref{quadratic-lemma} iii) reduces to the general case convergence
of Newton algorithms but on the subspace $\Omega$. However, if $\|\bx^*\|_0<s$,  there is no theoretic guarantee to find a common subspace that can identify the support of all $\bx^k$. Therefore, the quadratic convergence in iii) is different from  the general case convergence
of Newton algorithms and thus is non-trivial.
\end{remark}
Results in Theorem  \ref{quadratic-lemma} are established for sufficiently large $k$. In the sequel, we aim to show that such a $k$ has a finite lower bound. Then all results in Theorem  \ref{quadratic-lemma}  will hold for any $k$ greater than this bound.   To proceed with that, we need the following lemma with stronger assumptions.
\begin{lemma}\label{theorem-explicit-k0} Suppose $f$ is  $2s$-RSS  with $L_{2s}>0$ and  $2s$-RSC with $\ell_{2s}>0$. Let  $\{\bx^ k \}$ be the sequence generated by GPNP with setting $\sigma\in(0, 6/(\underline{\alpha}^2\ell_{2s}))$  and $\bx^*$ be its limit. Denote $\rho:=1- {\sigma \underline{\alpha}^2\ell_{2s}}/{6} \in(0,1)$. Then  for any $k\geq 0$, 
 \begin{eqnarray} \label{rate-xk-uk-x*}
 \arraycolsep=1.4pt\def\arraystretch{1.5}
\begin{array}{lll}
 \ell_{2s}\|\bx^k-\bx^*\|  &\leq &  2   \|\nabla_{T_k  \setminus T_*} f(\bx^*)\| +  \rho^{k/2}  \sqrt{2\ell_{2s}(f(\bx^{0}) - f(\bx^*))} ,\\
 \ell_{2s} \|\bu^k-\bx^*\| &\leq &   2  \|\nabla_{\Gamma_k  \setminus T_*} f(\bx^*)\| +  \rho^{k/2}  \sqrt{2\ell_{2s}(f(\bx^{0}) - f(\bx^*))},
\end{array}
\end{eqnarray}
where $T_k$ is any index set satisfying $T_k\supseteq \supp(\bx^k)$ and $|T_k|=s.$
\end{lemma}
\begin{proof} It is clear to see that $\rho \in(0,1)$ due to $\sigma\in(0, 6/(\underline{\alpha}^2\ell_{2s}))$. Let $T_k$ be an index set  satisfying
\begin{eqnarray}\label{def-Tk}
 \arraycolsep=1.4pt\def\arraystretch{1.5}
 \begin{array}{lllr}
T_k\supseteq \supp(\bx^k), ~~|T_k|=s.
 \end{array}
\end{eqnarray}
Since $2s$-RSC implies  $s$-RSC , $\bx^*$ is a unique local minimizer from Theorem \ref{global-convergence}, which by \eqref{l1-KKT} shows
 \begin{eqnarray}\label{cond-x*-1}
 \arraycolsep=1.4pt\def\arraystretch{1.5}
\begin{array}{lll}
\nabla_{T_*} f(\bx^{*}) = 0.
\end{array}
\end{eqnarray}
In addition, Lemma \ref{theo05} ii) states that $\{f(\bx^k)\}$ is a strictly decreasing sequence, and hence 
 \begin{eqnarray}\label{cond-x*-2}
 \arraycolsep=1.4pt\def\arraystretch{1.5}
\begin{array}{lll}
f(\bx^k) - f(\bx^{*}) \geq 0, \qquad \forall k\geq 0.
\end{array}
\end{eqnarray}
By $\bu^k=\bx^k(\alpha_k)$ and \eqref{uk-xk-Gammak},  we have
\begin{eqnarray}\label{xk1-xk}
 \arraycolsep=1.4pt\def\arraystretch{1.5}
 \bu^k-\bx^k =\bx^k(\alpha_k)-\bx^k=\left[ \begin{array}{c}
 -\alpha_k \nabla_{\Gamma_k} f(\bx^k)\\
-\bx^k_{\overline\Gamma_k}
 \end{array}\right]=\left[ \begin{array}{c}
 -\alpha_k \nabla_{\Gamma_k} f(\bx^k)\\
-\bx^k_{T_k\setminus\Gamma_k}\\
0
 \end{array}\right].
\end{eqnarray}
Moreover, as operator $\Pi_{s}$ keeps the $s$ largest elements (in magnitude) and $\bu^k=\bx^k(\alpha_k) \in  \Pi_{s}(\bx^{k}-\alpha_k \nabla f(\bx^k)),$  it has
\begin{eqnarray*} 
\forall~i\in\supp(\bu^k),\qquad  |(\bx^{k}-\alpha_k \nabla f(\bx^k))_i| \geq |(\bx^{k}-\alpha_k \nabla f(\bx^k))_j|,\qquad \forall~j\notin \supp(\bu^k).
\end{eqnarray*}
If $\|\bu^k\|_0=s$, then $\Gamma_k=\supp(\bu^k)$ due to $\Gamma_k\supseteq{\rm supp}(\bu^k)$ from \eqref{def-gamma-k}. Consequently, the above condition turns to
\begin{eqnarray} \label{i-j-gammak}
\forall~i\in\Gamma_k,\qquad  |(\bx^{k}-\alpha_k \nabla f(\bx^k))_i| \geq |(\bx^{k}-\alpha_k \nabla f(\bx^k))_j|,\qquad \forall~j\in \overline\Gamma_k 
\end{eqnarray}
If $\|\bu^k\|_0<s$, then  $\|\bx^{k}-\alpha_k \nabla f(\bx^k)\|_0<s$, this means $(\bx^{k}-\alpha_k \nabla f(\bx^k))_j=0$ for any $j\notin \supp(\bu^k)$. Therefore, \eqref{i-j-gammak} is still true for $\|\bu^k\|_0<s$ because of $\Gamma_k\supseteq{\rm supp}(\bu^k)$. It is noted that $|\Gamma_k| = |T_k| =s$ from \eqref{def-gamma-k} and \eqref{def-Tk}, thereby showing $|\Gamma_k\setminus T_k| = |T_k\setminus \Gamma_k|=s-|\Gamma_k\cap T_k|$.   This   allows  us to derive that
\begin{eqnarray*} 
 \arraycolsep=1.4pt\def\arraystretch{1.5}
 \begin{array}{lllll}
 \|(-\alpha_k \nabla f(\bx^k))_{\Gamma_k\setminus T_k}\|^2
 &=&\|(\bx^{k}-\alpha_k \nabla f(\bx^k))_{\Gamma_k\setminus T_k}\|^2&(\text{by $\eqref{def-Tk}$})\\
 &=&  \sum_{i\in \Gamma_k\setminus T_k}|(\bx^{k}-\alpha_k \nabla f(\bx^k))_i|^2   \\
 &  \geq&  \sum_{j\in T_k\setminus \Gamma_k}|(\bx^{k}-\alpha_k \nabla f(\bx^k))_j|^2&(\text{by $\eqref{i-j-gammak}$ and  $|\Gamma_k\setminus T_k| = |T_k\setminus \Gamma_k|$})\\
  &=& \|(\bx^{k}-\alpha_k \nabla f(\bx^k))_{T_k\setminus \Gamma_k}\|^2\\ 
 & \geq &  -2\|\bx^{k}_{T_k\setminus \Gamma_k}\|^2   +  (2/3)\|\alpha_k \nabla_{T_k\setminus \Gamma_k} f(\bx^k)\|^2,
\end{array}\end{eqnarray*}
where the last inequality used the fact  that $\|\bu-\bv\|^2\leq (1+1/t)\|\bu\|^2+(1+t)\|\bv\|^2$ for any $t>0$ and vectors $\bu,\bv$. Adding $\|(-\alpha_k \nabla f(\bx^k))_{\Gamma_k \cap  T_k}\|^2$ on the both sides of the above inequality yields
\begin{eqnarray*} 
 \arraycolsep=1.4pt\def\arraystretch{1.5}
 \begin{array}{lllll}
\|\alpha_k \nabla_{\Gamma_k} f(\bx^k)\|^2  
    \geq  -2\|\bx^{k}_{T_k\setminus \Gamma_k}\|^2 +(2/3) \|\alpha_k \nabla_{T_k} f(\bx^k)\|^2.
\end{array}\end{eqnarray*}
Now it follows from the above condition and \eqref{xk1-xk} that
\begin{eqnarray*} 
 \arraycolsep=1.4pt\def\arraystretch{1.5}
 \begin{array}{lllll}
\|\bu^k-\bx^k\|^2&=& \|\alpha_k \nabla_{\Gamma_k} f(\bx^k)\|^2+\|
\bx^k_{T_k\setminus\Gamma_k}\|^2\\
 &\geq& \|\alpha_k \nabla_{\Gamma_k} f(\bx^k)\|^2- (1/2)\|\alpha_k \nabla_{\Gamma_k} f(\bx^k)\|^2 +(1/3) \|\alpha_k \nabla_{T_k} f(\bx^k)\|^2\\
  &\geq& (\alpha_k^2/3)(\|  \nabla_{\Gamma_k} f(\bx^k)\|^2+ \|  \nabla_{T_k} f(\bx^k)\|^2)\\
  & \geq &  (\underline{\alpha}^2/3) \| \nabla_{\Gamma_k\cup T_k} f(\bx^k)\|^2,\qquad(\text{by $\eqref{alpha-k}$}) 
\end{array}\end{eqnarray*}
which together with \eqref{decrese-pro} gives rise to
 \begin{eqnarray}\label{decrese-pro-grad}
 \arraycolsep=1.4pt\def\arraystretch{1.5}
\begin{array}{lll}
	 \| \nabla_{\Gamma_k\cup T_k} f(\bx^k)\|^2	\leq  (3/{\underline{\alpha}^2})\|\bu^{k}-\bx^k\|^2 \leq  ({6}/({\sigma\underline{\alpha}^2}))(f(\bx^{k})  -f(\bx^{k+1})) .
\end{array}
\end{eqnarray}
Again by \eqref{i-j-gammak}, we can claim 
\begin{eqnarray*} 
 \arraycolsep=1.4pt\def\arraystretch{1.5}
\begin{array}{lll}
\forall~i\in\Gamma_k\setminus T_k,\qquad  |(-\alpha_k \nabla f(\bx^k))_i|&=&|(\bx^{k}-\alpha_k \nabla f(\bx^k))_i| \\
&\geq& |(\bx^{k}-\alpha_k \nabla f(\bx^k))_j|=|( \alpha_k \nabla f(\bx^k))_j|,\qquad \forall~j\in \overline\Gamma_k\setminus T_k. 
\end{array}
\end{eqnarray*}
This means the largest $| ( \nabla f(\bx^k))_i|, i\notin T_k$ are contained in $\Gamma_k\setminus T_k$. Therefore,   $\|  \nabla_{\Gamma_k\setminus T_k} f(\bx^k)\|^2 \geq \| \nabla_{T_*\setminus T_k} f(\bx^k)\|^2$ due to $|\Gamma_k|=s \geq |T_*|$.  Using this condition, $T_*=\supp(\bx^*), \Gamma_k\supseteq \supp(\bu^k)$,  and  $2s$-RSC with $\ell_{2s}>0$ suffices to
  \begin{eqnarray}\label{decrese-pro-1}
 \arraycolsep=1.4pt\def\arraystretch{1.5}
\begin{array}{lll}
f(\bx^*) &\geq& f(\bx^k)+\langle \nabla f(\bx^k), \bx^*-\bx^k \rangle	+(\ell_{2s}/2)\|\bx^k-\bx^*\|^2\\
& = & f(\bx^k)+\langle ( 1/\sqrt{\ell_{2s}}) \nabla_{T_k \cup T_*} f(\bx^k), \sqrt{\ell_{2s}}(\bx^*-\bx^k)_{T_k \cup T_*} \rangle+	(\ell_{2s}/2)\|(\bx^*-\bx^k)_{T_k \cup T_*}\|^2\\
& \geq & f(\bx^k)-  (1/\ell_{2s}) \|\nabla_{T_k \cup T_*} f(\bx^k)\|^2\\
& = & f(\bx^k)-  (1/\ell_{2s}) (\|\nabla_{T_k} f(\bx^k)\|^2+ \|\nabla_{T_*\setminus T_k} f(\bx^k)\|^2)\\
& \geq & f(\bx^k)-  (1/\ell_{2s}) (\|\nabla_{T_k} f(\bx^k)\|^2+ \|\nabla_{ \Gamma_k \setminus T_k} f(\bx^k)\|^2)\\
& = & f(\bx^k)-  (1/\ell_{2s}) (\|\nabla_{\Gamma_k\cup  T_k} f(\bx^k)\|^2\\
& \geq & f(\bx^k)-  (6/(\sigma \underline{\alpha}^2\ell_{2s}))(f(\bx^{k})  -f(\bx^{k+1})),\qquad (\text{by \eqref{decrese-pro-grad}})
\end{array}
\end{eqnarray}
which after simple manipulation displays
 \begin{eqnarray}\label{decrese-pro-grad-0}
 \arraycolsep=1.4pt\def\arraystretch{1.5}
\begin{array}{lll}
f(\bx^{k+1}) - f(\bx^*) \leq \rho (f(\bx^{k}) - f(\bx^*)) \leq \cdots \leq \rho^{k+1} (f(\bx^{0}) - f(\bx^*)) 
\end{array}
\end{eqnarray}
due to \eqref{cond-x*-2}.  
Again $2s$-RSC with $\ell_{2s}>0$ suffices to
  \begin{eqnarray} \label{f-rate-xk-x*}
 \arraycolsep=1.4pt\def\arraystretch{1.5}
\begin{array}{lll}
f(\bx^k) &\geq& f(\bx^*)+\langle \nabla f(\bx^*), \bx^k-\bx^* \rangle	+(\ell_{2s}/2)\|\bx^k-\bx^*\|^2\\
&=& f(\bx^*)+\langle   \nabla_{T_k \cup T_*} f(\bx^*),   (\bx^k-\bx^*)_{T_k \cup T_*} \rangle	+(\ell_{2s}/2)\|\bx^k-\bx^*\|^2\\
& \geq & f(\bx^*)-  \|\nabla_{T_k \cup T_*} f(\bx^*)\| \cdot\|   \bx^k-\bx^*\|+ (\ell_{2s}/2)\|\bx^k-\bx^*\|^2\\
& = & f(\bx^*)-  \|\nabla_{T_k  \setminus T_*} f(\bx^*)\|\cdot\|   \bx^k-\bx^*\|+ (\ell_{2s}/2)\|\bx^k-\bx^*\|^2,\qquad (\text{by \eqref{cond-x*-1}})
\end{array}
\end{eqnarray}
which together with \eqref{decrese-pro-grad-0} immediately brings 
 \begin{eqnarray} \label{rate-xk-x*}
 \arraycolsep=1.4pt\def\arraystretch{1.5}
\begin{array}{lll}
\|\bx^k-\bx^*\| &\leq & \frac{1}{\ell_{2s}}\left( \|\nabla_{T_k  \setminus T_*} f(\bx^*)\| + \sqrt{ \|\nabla_{T_k  \setminus T_*} f(\bx^*)\|^2+2\ell_{2s}(f(\bx^k) - f(\bx^*))}\right) \\
&\leq & \frac{1}{\ell_{2s}}\left(  2\|\nabla_{T_k  \setminus T_*} f(\bx^*)\|  + \sqrt{ 2\ell_{2s}(f(\bx^k) - f(\bx^*))}\right) \\
&\leq &\frac{2 }{\ell_{2s}}\|\nabla_{T_k  \setminus T_*} f(\bx^*)\| + \frac{\sqrt{2\ell_{2s}(f(\bx^{0}) - f(\bx^*))}}{ {\ell_{2s}}}  \rho^{k/2}.\qquad (\text{by \eqref{decrese-pro-grad-0}})
\end{array}
\end{eqnarray}
Similar reason to show \eqref{f-rate-xk-x*}  also allows us to obtain  \begin{eqnarray*}  
 \arraycolsep=1.4pt\def\arraystretch{1.5}
\begin{array}{lll}
f(\bu^k)  
& \geq & f(\bx^*)-  \|\nabla_{\Gamma_k  \setminus T_*} f(\bx^*)\|  \|\bu^k-\bx^*\|+ (\ell_{2s}/2)\|\bu^k-\bx^*\|^2,
\end{array}
\end{eqnarray*}
which by \eqref{rate-xk-x*} and $f(\bx^k)\geq f(\bu^k)$ from \eqref{fact-u-z-0} suffices to
 \begin{eqnarray} \label{rate-xk-u*}
 \arraycolsep=1.4pt\def\arraystretch{1.5}
\begin{array}{lll}
\|\bu^k-\bx^*\| \leq    \frac{2 }{\ell_{2S}}\|\nabla_{\Gamma_k  \setminus T_*} f(\bx^*)\| + \frac{\sqrt{2\ell_{2s}(f(\bx^{0}) - f(\bx^*))}}{ {\ell_{2s}}}  \rho^{k/2}.
\end{array}
\end{eqnarray}
Combining \eqref{rate-xk-x*} and \eqref{rate-xk-u*} finishes the whole proof.
\end{proof} 
Based on Lemma \ref{theorem-explicit-k0} and Theorem \ref{quadratic-lemma}, we can obtain the following corollary, where the assumptions are similar to  those in \cite[Corollary 3]{yuan2017gradient}.
\begin{corollary}Suppose $f$ is  $2s$-RSS  with $L_{2s}>0$ and  $2s$-RSC with $\ell_{2s}>0$. Let  $\{\bx^ k \}$ be the sequence generated by GPNP with setting $\sigma\in(0, \min\{\ell_{s}/4, 6/(\underline{\alpha}^2\ell_{2s})\})$. Denote 
 \begin{eqnarray} \label{k-lower-bd}
 \arraycolsep=1.4pt\def\arraystretch{1.5}
\begin{array}{llllll}
\qquad k(c) :=  \left \lfloor \log_\rho\left( \frac{c^2  }{2\ell_{2s}(f(\bx^{0}) - f(\bx^*))} \right) \right\rfloor, ~~c_1:=\ell_{2s}\min_{i\in T_*}|x_i^*|-2\|\nabla_{\overline T_*} f(\bx^*)\|_\infty,~~ c_2:= \min\left\lbrace \frac{\ell_{2s}^2}{2M_s^*},\frac{\ell_{2s}^3}{8L_{2s}M_s^*}, \frac{\epsilon\ell_{2s}}{L_{2s}}\right\rbrace,
\end{array}
\end{eqnarray}
where  $ \lfloor t \rfloor$ counts the largest integer no greater than $t+1$. Assume $c_1>0$. Then the following statements are true.
\begin{itemize}
\item[1)] The support set of $\bx^*$ can be identified by (\ref{support-identify}) for any $k>k(c_1)$; 
\item[2)] If we further assume $f$ is locally $s$-RHLC at $ \bx^*$  with $M_s^*>0$. Then for any $k>\max\{k(c_1),k(c_2)\}$, the Newton step is always admitted,  and the sequence satisfies (\ref{quadratic-rate-formula}) and thus converges to $\bx^*$ quadratically.  
\end{itemize}
\end{corollary}
\begin{proof}
1) Suppose the support set of $\bx^*$ is not identified by \eqref{support-identify} when $k>k(c_1)$, that is, $T_*\setminus T_k\neq \emptyset$ or $T_*\setminus \Gamma_k\neq \emptyset$. We only consider case $T_*\setminus T_k\neq \emptyset$ as the proof for the latter is similar. Direct calculation yields that
 \begin{eqnarray*}
 \arraycolsep=1.4pt\def\arraystretch{1.5}
\begin{array}{llll}
c_1 &=&
\ell_{2s}\min_{i\in T_*}|x_i^*| -2\|\nabla_{\overline T_*} f(\bx^*)\|_\infty \\
&\leq&
\ell_{2s}\min_{i\in T_*\setminus T_k}|x_i^*| -2\|\nabla_{T_k\setminus T_*} f(\bx^*)\|_\infty \\
&\leq&   \frac{\ell_{2s}}{\sqrt{|T_*\setminus T_k|}}\|(\bx^k-\bx^*)_{T_*\setminus T_k}\| - \frac{2}{\sqrt{|T_k\setminus T_*|}} \|\nabla_{T_k \setminus T_*} f(\bx^*)\| & (\text{by \eqref{def-Tk}}) \\
&=&  \frac{1}{\sqrt{s-|T_k \cap T_*|}}   (\ell_{2s}\|(\bx^k-\bx^*)_{T_*\setminus T_k}\|-2\|\nabla_{T_k \setminus T_*} f(\bx^*)\|) \\
&\leq&  \ell_{2s}  \|(\bx^k-\bx^*)_{T_*\setminus T_k}\|-2\|\nabla_{T_k \setminus T_*} f(\bx^*)\| \\
&\leq&  \ell_{2s}  \| \bx^k-\bx^*\|-2\|\nabla_{T_k \setminus T_*} f(\bx^*)\| \\
&\leq &  \sqrt{2\ell_{2s}(f(\bx^{0}) - f(\bx^*))} \rho^{k/2}& (\text{by  \eqref{rate-xk-uk-x*}})\\
& < &  \sqrt{2\ell_{2s}(f(\bx^{0}) - f(\bx^*))} \rho^{k(c_1)/2} \leq c_1.& (\text{by $k>k(c_1)$}) 
\end{array}
\end{eqnarray*}
This is a contradiction and shows the conclusion.\\
2) Note that  $2s$-RSC with $\ell_{2s}>0$ implies  $s$-RSC with $\ell_{s}>0$ and $\ell_{s} \geq \ell_{2s}>0$. Thus all conditions in Theorem \ref{quadratic-lemma} hold.    To prove the results, we will show that for any $k>\max\{k(c_1),k(c_2)\}$, (i) condition \eqref{H-u-bounded} holds, (ii) conditions \eqref{Newton-switch-on} for switching on Newton steps are satisfied, and (iii) condition \eqref{v-u-0-0} holds.\\
(i) If $\|\bx^*\|_0<s$ then $\nabla f(\bx^{*})  =0$ by \eqref{l1-KKT}. If $\|\bx^*\|_0=s$ then claim 1) and (\ref{support-identify}) suffices to $T_*=\supp(\bx^*)\equiv \supp(\bx^k)\equiv \Gamma_k$ for any $k>k(c_1)$. Therefore, both cases indicates $\|\nabla_{T_k  \setminus T_*} f(\bx^*)\| = \|\nabla_{\Gamma_k  \setminus T_*} f(\bx^*)\|=0$ for any $k>k(c_1)$. Now, it follows from \eqref{rate-xk-u*} that, for any $k>\max\{k(c_1),k(c_2)\}$,
\begin{eqnarray} \label{u-x-kc2}
 \arraycolsep=1.4pt\def\arraystretch{1.5}
\begin{array}{lll}
\|\bu^k-\bx^*\| &\leq&    \frac{\sqrt{2\ell_{2s}(f(\bx^{0}) - f(\bx^*))}}{ {\ell_{2s}}}  \rho^{k/2} \\
&\leq&     \frac{\sqrt{2\ell_{2s}(f(\bx^{0}) - f(\bx^*))}}{ {\ell_{2s}}}  \rho^{k(c_2)/2} \\
&\leq&  \min\left\lbrace \frac{\ell_{2s}}{2M_s^*},\frac{\ell_{2s}^2}{8L_{2s}M_s^*}, \frac{\epsilon }{L_{2s}} \right\rbrace \leq \frac{\ell_{2s}}{2M_s^*},
\end{array}
\end{eqnarray} 
which allows us to derive \eqref{H-u-bounded}. \\
(ii) If follows from the first condition in \eqref{u-x-kc2} that for any $k>\max\{k(c_1),k(c_2)\}$, 
 \begin{eqnarray*}
 \arraycolsep=1.4pt\def\arraystretch{1.5}
\begin{array}{lllll}
\|\bg^k\| =\|\bg^k -\nabla f(\bx^{*})  \| 
& \leq&  L_{2s} \|\bu^k - \bx^*\| ~~&(\text{by $2s$-RSS})\\
&\leq&  L_{2s}  \frac{\epsilon }{L_{2s}} <\epsilon,  &(\text{by  \eqref{u-x-kc2}}) 
\end{array}
\end{eqnarray*}
which means  the second condition  in (\ref{Newton-switch-on}) is satisfied.\\
(iii) For any $k>\max\{k(c_1),k(c_2)\}$, we have the following inequalities
\begin{eqnarray} \label{vk-uk-l-M}
 \arraycolsep=1.4pt\def\arraystretch{1.5}
\begin{array}{lllll} 
  \|\bv^{k} -\bu^{k}\| 
 &\leq& (2L_{2s}/\ell_s)\| \bu^k-\bx^*\|\leq \frac{\ell_{2s}}{4M_s^*} .&\qquad(\text{by \eqref{v-u-0} and \eqref{u-x-kc2}})  
\end{array}
\end{eqnarray} 
Now, let  $\bu_t^k:=t\bu^k+(1-t)\bv^k$ with some $t\in(0,1)$. Then $\|\bu_t^k-\bu^k\|\leq (1-t)\|\bu^k-\bv^k\| \leq \|\bu^k-\bv^k\| \leq \frac{\ell_{2s}}{4M_s^*} \leq \frac{\ell_{s}}{2M_s^*}$. This means that $\bu_t^k,\bu^k\in S\cap N(\bx^*, \frac{\ell_{s}}{2M_s^*})$,  which by $s$-RHLC  and \eqref{vk-uk-l-M} renders that
\begin{eqnarray} \label{vk-uk-bd}
 \arraycolsep=1.4pt\def\arraystretch{1.5}
\begin{array}{lll}
\|\nabla^2_{\Gamma_k\Gamma_k}f(\bu_t^k)  -\H^k\| &\leq&  M_s^* \|\bu_t^k-\bu^k\|\leq  M_s^* \|\bv^k-\bu^k\| \leq \frac{\ell_{2s}}{4} \leq \frac{\ell_{s}}{4}.
\end{array}
\end{eqnarray} 
It follows from the Mean Value Theorem  that 
  for any $k>k(c_2)$,
\begin{eqnarray*}
 \arraycolsep=1.4pt\def\arraystretch{1.5}
\begin{array}{lllr}
2f(\bv^{k})-2f(\bu^{k})
&=& 2\langle  \bg^{k}, \bv^{k} -\bu^{k}\rangle+  \langle \nabla^2f(\bu_t^{k}) (\bv^{k} -\bu^{k}), \bv^{k} -\bu^{k} \rangle\\
&=& 2\langle \bg^{k}_{\Gamma_k}, (\bv^{k} -\bu^{k})_{\Gamma_k} \rangle+  \langle \H^{k} (\bv^{k} -\bu^{k})_{\Gamma_k}, (\bv^{k} -\bu^{k})_{\Gamma_k} \rangle\\
& + & \langle  (\nabla^2_{\Gamma_k\Gamma_k} f(\bu_t^{k})-\H^{k}) (\bv^{k} -\bu^{k})_{\Gamma_k}, (\bv^{k} -\bu^{k})_{\Gamma_k} \rangle  \\
&\leq& -\langle \H^{k} (\bv^{k} -\bu^{k})_{\Gamma_k}, (\bv^{k} -\bu^{k})_{\Gamma_k} \rangle +  ({\ell_{s}}/{4}) \|(\bv^{k} -\bu^{k})_{\Gamma_k}\|^2 &(\text{by \eqref{Newton-descent-property}, \eqref{vk-uk-bd}})\\
&\leq&  - (\ell_s/2) \|(\bv^{k} -\bu^{k})_{\Gamma_k}\|^2 + ({\ell_{s}}/{4}) \|(\bv^{k} -\bu^{k})_{\Gamma_k}\|^2 &(\text{by \eqref{H-u-bounded}}) \\
&=&  - (\ell_s/4) \| \bv^{k} -\bu^{k}\|^2 ~&(\text{by \eqref{Gv-Gu}})\\
&\leq& -\sigma \|\bv^{k} -\bu^{k}\|^2. &(\text{by   $0<\sigma\leq\ell_s/4$})
\end{array}
\end{eqnarray*}
Therefore, condition \eqref{v-u-0-0} holds for $k>\max\{k(c_1),k(c_2)\}$.\\
The rest proof is omitted as it is the same as that for Theorem \ref{quadratic-lemma} ii) and  iii). 
\end{proof}
\subsection{Application to CS}
To end this section, we would like to see the performance of GPNP for  CS problems with the objective function,
\begin{equation*}
f(\bx)= f_{cs}(\bx)=({1}/{2}) \|\A\bx-\bb\|^{2}.
\end{equation*} 
 The following corollary shows that all previously established convergence results  are valid for CS problems if we are given a mild assumption on $\A$. {The assumption is the so-called $s$-regularity (see \cite[Definition 2.2]{Beck13}). We say matrix $\A$ is $s$-regular if any its $s$ columns are linearly independent. It is easy to see that if $\A$ is $s$-regular, then \begin{eqnarray}\label{RIP-lambda}
\lambda_s:=  \min\left\{\lambda_{\min} (\A_{:T}^\top \A_{:T}): |T|=s\right\} >0.
\end{eqnarray} }
\begin{corollary}\label{cor-cs} Let  $\{\bx^ k \}$ be the sequence generated by GPNP for (\ref{sco}) with $f=f_{cs}$. {Suppose $\A$ is $s$-regular. Then }
for sufficiently large $k$, the following statements hold.
\begin{itemize}
\item[1)] The whole sequence converges to a unique local minimizer (say $\bx^*$). 
\item[2)] The support set of $\bx^*$ can be identified by (\ref{support-identify}).
\item[3)]  Newton steps are always admitted if we set $\sigma\in(0,\lambda_s/4]$.
\item[4)] $\bx^{k}=\bx^{*}$ {after finitely many $k$.}
\end{itemize}
\end{corollary}
\begin{proof} We note that $f_{cs}$ is $2s$-RSS with $L_{2s}=\lambda_{\max}(\A^\top \A)$ and $s$-RSC with $\ell_s=\lambda_s$ if $\lambda_s>0$. Hence,  1) - 3) in Theorems \ref{global-convergence} and \ref{quadratic-lemma} are valid for CS immediately. We now prove 4). It follows conditions \eqref{l1-KKT} and \eqref{support-identify}  that
\begin{eqnarray}\label{b*-gamma-1}
\nabla_{\Gamma_k}f_{cs}(\bx^*)= \A^\top_{:\Gamma_k}(\A\bx^*-\bb)=  \A^\top_{:\Gamma_k}(\A_{:\Gamma_k}\bx^*_{\Gamma_k}-\bb)=0 
\end{eqnarray}	
for sufficiently large $k$. Moreover,  it is easy to see that  Newton  step  \eqref{Newton-descent-property} turns to
\begin{eqnarray}\label{Newton-descent-property-cs}
	  \H^k \bv^{k}_{\Gamma_k} =   \A_{:\Gamma_k}^\top \bb, ~~\bv^{k}_{\overline\Gamma_k}=0. 
\end{eqnarray}
where $\H^k:=\A_{:\Gamma_k}^\top \A_{:\Gamma_k}$.  The above two conditions as well as 2) that the Newton step  is always taken, namely, $\bx^{k+1}=\bv^{k}$, for sufficiently large $k$, result in the following chain of equalities,
\begin{eqnarray*} 
	 \arraycolsep=1.4pt\def\arraystretch{1.5}
\begin{array}{lllr}
\|\bx^{k+1}-\bx^*\| &=& \|\bv^{k}-\bx^*\| = \|\bv^{k}_{\Gamma_k}-\bx^*_{\Gamma_k}\| &(\text{by   \eqref{Gv-Gu}})\\
&=& \|(\H^k)^{-1}(\H^k\bv^{k}_{\Gamma_k}-\H^k\bx^*_{\Gamma_k} )\|&(\text{by $\lambda_s>0$}) \\
&=& \|(\H^k)^{-1}(\A_{:\Gamma_k}^\top \bb-\A_{:\Gamma_k}^\top \bb )\| &(\text{by  \eqref{Newton-descent-property-cs} and \eqref{b*-gamma-1}})\\
&=&0,
\end{array}
\end{eqnarray*}
 which draws the conclusion immediately. 
\end{proof}

\begin{remark}\label{remark-ass-cs}
Regarding Corollary \ref{cor-cs}, we have the following observations.
\begin{itemize}
\item[I)]  One can discern that $\lambda_s >0$ is a weaker condition than RIP \cite{candes2005decoding} associated with an $s$th order RIC $\delta_s:=\delta_s(\A)$ of $\A$ defined as the smallest positive constant $\delta$ such that  $(1-\delta)\|\bx\|^2\leq \|\A\bx\|^2\leq(1+\delta)\|\bx\|^2$
for all $s$-sparse vectors $\bx$. It is easy to check that $\lambda_s \geq  1-\delta_s.$
 Therefore, if $\delta_s<1$ then $\lambda_s>0$. This means if  matrix $\A$ satisfies RIP with RIC $\delta_s<1$, then GPNP converges globally and terminates within finitely many steps. We note that condition $\delta_s<1$ is weaker than those used in   \cite{blumensath2009iterative, zhao2020improved} for IHT,  \cite{foucart2011hard} for HTP,  \cite{dai2009subspace} for SP,  and \cite{needell2009cosamp,foucart2017mathematical, zhao2020improved} for CoSaMP.
 
\item[II)] It is worth mentioning that NHTP proposed in \cite{zhou2021global} also has the global and quadratic convergence properties under assumptions similar to those in Theorem \ref{quadratic-lemma}. However, when it comes to CS problems, the required assumptions indicate $\delta_{2s}<1/3$ and a number of parameters that should be chosen carefully from proper ranges (see \cite[Corollary 11]{zhou2021global}). However,  GPNP needs a weaker assumption (e.g., $\delta_{s}<1$) and its associated parameters can be chosen more flexibly.  
 
\item[III)]Finally, it is known that Newton-type methods for solving the unconstrained strongly quadratic programming (USQP) can terminate in one step if the starting point is close enough to the optimal solution.  From 4) that GPNP terminates within finitely many steps, which is a better result than quadratic convergence property, this can be regarded as an extension of Newton-type methods for solving the USQP. 

\end{itemize}   
\end{remark}

\begin{remark}We have some comments on the computational complexity of GPNP for solving  CS problems. Note that the calculations of $\Pi_{s}(\cdot)$ and the gradient dominate the computation for each step of gradient descent. Their total computational complexity is about $O(mn)$. 
 The worst-case computational complexity of deriving $\bv^{k}$ via (\ref{Newton-descent-property-cs}) is about $O(s^3+ms^2)$. Overall, the entire computational complexity of the $k$th iteration of Algorithm \ref{algorithm 1} is $$O(s^3+ms^2 + q_kmn),$$  
 where $q_k$ is the smallest integer  such that  (\ref{armijio-descent-property}) and $q_k\leq \lfloor \log_{\gamma}(\underline\alpha/\tau) \rfloor$ if $f$ is $2s$-RSS. Such a computational complexity is considerably low due to $s\ll n$ and $m< n$.
\end{remark}
\begin{table}[H]
	\renewcommand{\arraystretch}{1.1}\addtolength{\tabcolsep}{0pt}
	\caption{Convergence rate of different algorithms.}\vspace{-5mm}
	\label{tab:con-rate}
	\begin{center}
		\begin{tabular}{lllllll }
			\hline
			& & \multicolumn{2}{c}{SCO} && \multicolumn{2}{c}{CS}\\\cline{3-4}\cline{6-7}
	Algs. & Ref. & Rate & Assumptions && Rate & Assumptions\\\hline		
	SP&\cite{dai2009subspace}&$--$&$--$&&TFMS&RIP\\
AIHT&\cite{blumensath2012accelerated} &$--$&$--$&&TFMS&RIP\\
NIHT&\cite{blumensath2010normalized} &$--$&$--$&&TFMS&RIP\\
HTP$^\mu$&\cite{foucart2011hard}&$--$&$--$&&TFMS&RIP or $\|A\|^2<1/\mu$\\
FHTP$^\mu$&\cite{foucart2011hard}&$--$&$--$&&TFMS&RIP or $\|A\|^2<1/\mu$\\
GraHTP& \cite{yuan2017gradient}&Linear&$2s$-RSS, $2s$-RSC&&$--$&$--$\\
NHTP&\cite{zhou2021global}&Quadratic&$2s$-RSS, $s$-RSC&&TFMS&$2s$-RSS, $s$-RSC\\
GPNP& &Quadratic&$2s$-RSS, $s$-RSC&&TFMS&$s$-regularity\\\hline
	IRLSLq&\cite{lai2013improved}&$--$&$--$&&Superlinear&RIP\\
SAMP&\cite{do2008sparsity} &$--$&$--$&&TFMS&RIP\\\hline
		\end{tabular}
	\end{center}
\end{table}
\begin{remark}Finally, we summarize the convergence rates of some state-of-the-art algorithms in Table \ref{tab:con-rate}, where we have four types of convergence rates: linear rate, superlinear rate, quadratic rate, and termination within finitely many steps (TFMS). For HTP, there is a step size $\mu>0$. One can find that for general problem (\ref{sco}), quadratic convergence has been established  only for NHTP and GPNP. For CS problems, most of the algorithms can terminate within finitely many steps under RIP conditions.
\end{remark}

\section{Numerical Experiments}\label{sec:num}
{In this section, we conduct  extensive numerical experiments}  to showcase the performance of  GPNP (available at \href{https://github.com/ShenglongZhou/GPNP}{https://github.com/ShenglongZhou/GPNP})   by using MATLAB (R2019a) on a laptop with  $32$GB memory and  2.3Ghz CPU.  Parameters are set as follows: $\bx^{0}=0$ for CS problems and $\bx^{0}=1$ for QCS problems, { where $\bx^{0}=0$  (or $\bx^{0}=1$) means that all elements in $\bx$ are 0 (or 1).} Let $ \tau =5, \sigma=10^{-4},  \gamma=0.5$, $\epsilon=0.01$, and $\varepsilon=10^{-5}$. Besides halting condition $\pi_k\leq\varepsilon$, we also terminate GPNP if the maximum number of iterations ii over 5000.
\subsection{Solving CS problems}
We organize this part as follows. Two testing examples are first described, followed by selections of nine greedy methods and nine relaxation methods. Then extensively numerical comparisons among GPNP and the benchmarked methods are provided.
\subsubsection{Testing examples}\label{sec:exs}
We will solve CS problems with two types of measurement matrices $\A$: the randomly generated data and the 2-dimensional image data. For the former, we consider exact recovery $\bb=\A \bx$, where sensing matrix $\A$  chosen as in \cite{wang2013l1,yin2015minimization,zhou2016null}, while for the latter, we consider inexact recovery $\bb=\A \bx +  \boldsymbol \zeta$, where $\boldsymbol \zeta$ is the noise and $\A$ will be described in the sequel.

\begin{example}[Gaussian matrix]\label{cs-ex} Let $\A\in\mathbb{R}^{m\times n}$  be a random Gaussian matrix  with  entries being identically and independently distributed  (i.i.d.) samples  from standard normal distribution $\mathcal{N}(0,1)$. We then normalize each column to {have  unit length.} Next,  $s$ nonzero components of  `ground  truth' signal $\bx^*$ are also i.i.d. samples from  $\mathcal{N}(0,1)$, and their indices are picked randomly. Finally, the observation vector is given by $\bb=\A \bx^*$. 
\end{example}

\begin{example}[Image data]\label{cs-imag} Some images are naturally non-sparse but can be sparse under some wavelet transforms. Here, we take advantage of the Daubechies wavelet 1, denoted as $\bf W(\cdot)$.  Then the images under this transform, i.e., $\bx^*:=\bf W({\boldsymbol \omega})$, are sparse,  and ${\boldsymbol \omega}$ is the vectorized intensity of an input image. Therefore, the explicit
form of the sampling matrix may not be available. We consider a sampling matrix taking form $\A=\bf F \bf W^{-1}$, where $\bf F$ is the partial fast Fourier transform, and $\bf W^{-1}$ is the inverse of $\bf W$. The observation  vector is given by $\bb=\A \bx^*+  \boldsymbol \zeta$, where noise $\boldsymbol \zeta$ has each element $\zeta_i\sim {\tt nf }\cdot\mathcal{N}(0,1)$ and ${\tt nf }$ is the noise factor which is set as $0.05$ or $0.1$.   We will compute a gray image, see  Figure   $\ref{fig:gray-img}$, with size $256\times256(=65536=n)$ and a color image, see   Figure   $\ref{fig:color-img}$, with size $3\times 256\times256 (=196608=n)$. Both have the sampling size $m=9793$. {The signal to
noise ratio,  $\|\A \bx^*\|/ \|\boldsymbol \zeta\|$, is  $13.53$ when ${\tt nf }=0.05$ and $6.77$ when ${\tt nf }=0.1$ for grey image, and  $8.32$ when ${\tt nf }=0.05$ and $4.16$ when ${\tt nf }=0.1$ for color image.}
\end{example}

To measure the performance of one method, we report the computational time (in seconds), the relative error (ReEr) 
  for Example \ref{cs-ex}, and the
peak signal to noise ratio (PSNR) for Example \ref{cs-imag}, where $${\rm ReEr}:=  {\|\bx-\bx^*\|}\cdot{\|\bx^*\|^{-1}},~~~~{\rm PSNR} :=10\log_{10}\left( {n}{\|\bx-\bx^*\|^{-2}}\right)$$
and $\bx$ is the solution obtained by the method. We say a recovery is successful if {${\rm ReEr}<10^{-4}$.}

\subsubsection{Some insight}
{GPNP has two conditions, namely, \eqref{Newton-switch-on}, for switching on Newton steps. Therefore, we would like to see how many steps at which they are satisfied. To proceed with that, we fix $m=64,n=256$ and $s=5$, or $10$ for Example \ref{cs-ex} and report average results over 200 instances in Figure \ref{fig:effect-eps}, where NoGV, NoSI, NoNP and NoIT represent the number of gradient vanishing (i.e., second condition in \eqref{Newton-switch-on}), the number of support identified (i.e., first condition in \eqref{Newton-switch-on}), the number of Newton pursuit, and the number of iterations, respectively. As expected, the larger $\epsilon$, the more NoGV and NoNP, but the fewer NoSI and NoIT. This is because more Newton steps would accelerate the convergence and hence reduce NoIT. 
}

\begin{figure}[!th]
\begin{subfigure}{.48\textwidth}
	\centering
	\includegraphics[width=1\linewidth]{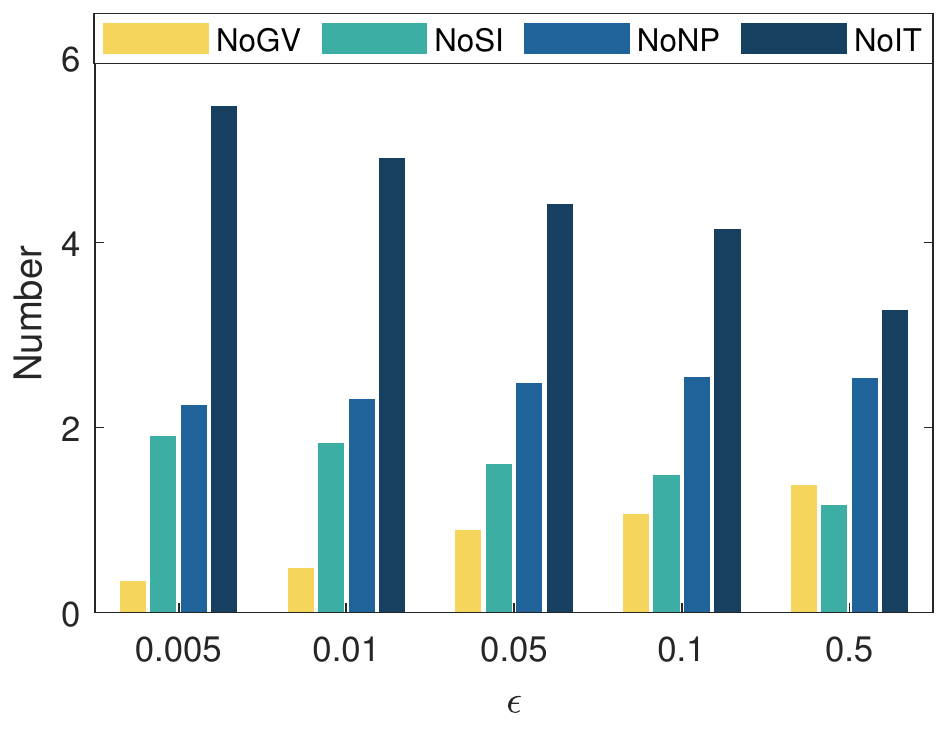}
	\caption{$s=5$.}
	\label{fig:no-5}
\end{subfigure}	 
\begin{subfigure}{.48\textwidth}
	\centering
	\includegraphics[width=1\linewidth]{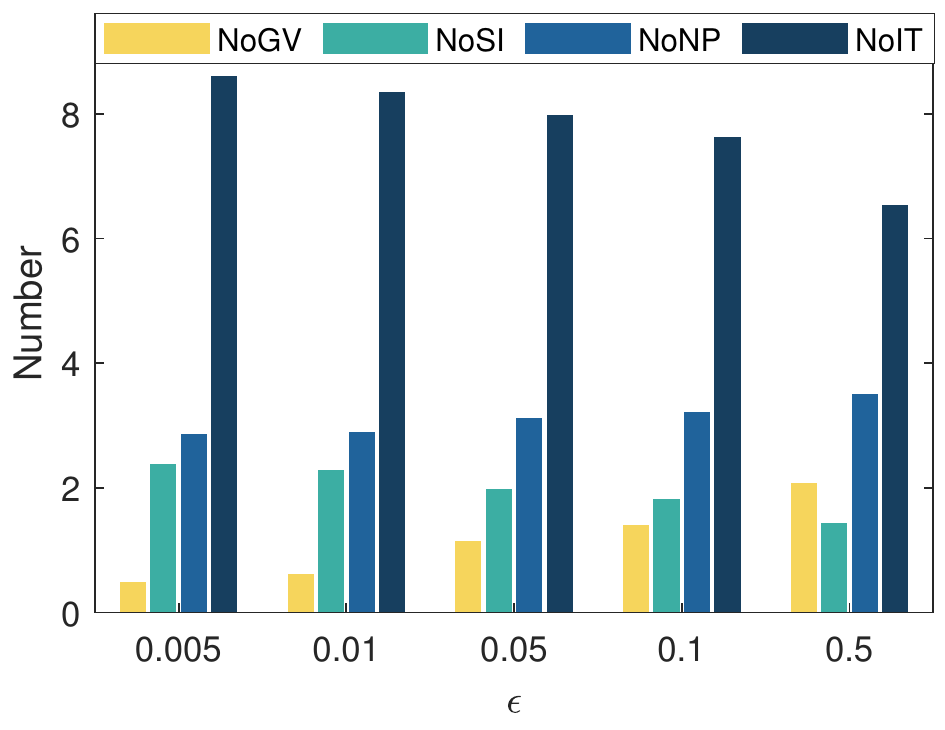}
	\caption{$s=10$.}
	\label{fig:no-10}
\end{subfigure}\vspace{-2mm}
\caption{Effect of  $\epsilon$.\label{fig:effect-eps}}
\end{figure}

\subsubsection{Benchmark methods}
There is a huge number of algorithms that have been proposed for CS problems. Many of them are available at \url{https://sites.google.com/site/igorcarron2/cs#reconstruction} or the authors' homepages.  We select nine greedy  and nine relaxation methods summarized in Table \ref{tab:algs}. To conduct fair  comparisons, the initial points for all methods are set as $\bx^0=0$. Moreover, besides the default halting conditions, to accelerate the termination of all methods, we also set additional ones as follows: stop a method if at $k$th iteration it meets $\|\A\bx^{k+1}-\bb\|<10^{-8}\|\bb\|$ or $k>1000$ for Example  \ref{cs-ex}, and ${\rm std}(f_k,f_{k-1},\ldots ,f_{k-5})<10^{-5}(1+f_k)$ or $k>100$ for Example \ref{cs-imag}. Finally, all other parameters are set as ones described in Table \ref{tab:algs}.

\begin{table}[!th]
	\renewcommand{\arraystretch}{1.1}\addtolength{\tabcolsep}{6pt}
	\caption{Parameters of benchmark methods.}\vspace{-2mm}
	\label{tab:algs}
	\begin{center}
		\begin{tabular}{llllllll }
			\hline
			\multicolumn{3}{c}{Greedy methods}&&\multicolumn{3}{c}{Relaxation methods}\\\cline{1-3}\cline{5-7}
			Algs. & Ref. & Pars.&&Algs. & Ref. & Pars.\\\hline
		GP&\cite{blumensath2008gradient} &stopTol$=s$&	&YALL1&\cite{jfyang2011alternatingdirectionalgorithmsfor}&${\rm tol}=10^{-6}$\\
		SP&\cite{dai2009subspace}&default&&L1Magic&\cite{candes2006robust,candes2006near,candes2006stable}&${\rm tol}=10^{-6}$\\ 
		OMP&\cite{pati1993orthogonal,tropp2007signal} &default&&SPGL1&\cite{van2009probing,van2011sparse}&${\rm tol}=10^{-6}$\\
		AIHT&\cite{blumensath2012accelerated} &${\rm thresh}=10^{-10}$&&IRL1&\cite{candes2008enhancing}&${\rm tol}=10^{-6}$\\
		NIHT&\cite{blumensath2010normalized} &default&&DCL1L2&\cite{yin2015minimization}&$\lambda=10^{-6},\delta=10^{-5}$\\
		{CoSaMP}&\cite{needell2009cosamp} &default&&IRLSLq&\cite{lai2013improved}&$q=0.5,\lambda=10^{-6}$\\
		ALPS&\cite{kyrillidis2011recipes} &default&&AMP&\cite{donoho2009message}&${\rm tol}=10^{-6}$\\
		FHTP&\cite{foucart2011hard}&default&&GPSR&\cite{figueiredo2007gradient}&$\tau=10^{-3}\|\A^\top \bb\|_{\infty}$
		\\
		\multirow{2}{*}{GraSP}&\multirow{2}{*}{\cite{bahmani2013greedy}} &${\rm tolF}=10^{-8}$ &&\multirow{2}{*}{SAMP}&\multirow{2}{*}{\cite{do2008sparsity}}&${\rm step\_size}=5$\\
 & &  ${\rm tolG} = 10^{-4}$&& & &$ \sigma=10^{-6}$\\		
		\hline  
		\end{tabular}
	\end{center}
\end{table}

	It is worth mentioning that SP, CoSaMP, SAMP, and our method GPNP belong to the family of second-order methods that exploit the information of Hessian matrix $\A^\top \A$. Therefore, it is naturally expected that they would produce solutions with higher accuracy. Moreover,  relaxation methods have been proposed to solve the relaxation models of the original CS problems.  For example, methods including YALL1, L1Magic, SPGL1, GPSR, and AMP have been developed to address the $\ell_1$-norm involved relaxations, while the others have been cast to deal with more advanced relaxation models.

\subsubsection{Numerical comparisons}
{\bf a) Greedy methods solving Example \ref{cs-ex}.} To see the performance of greedy methods, we begin with running $500$ independent trials with fixed $n = 256, m = 64$ and recording the corresponding success rates (which is defined by the percentage of the number of successful recoveries over all trials) at sparsity levels $s$ from $5$ to $35$. As shown in  Figure   \ref{fig:succ-s-ex1-greedy}, the bigger $s$ results in lower success rates, namely, harder recoveries. One can observe that GPNP has an outstanding recovery performance. For example, when $s = 25$, it still obtains $95\%$
successful recoveries while the other methods only guaranteed less than $15\%$ successful ones. The improvement is significant. 

We then fix  $n = 256, s = 13$ and alter $m$ from $\lfloor0.08n\rfloor$ to $\lfloor0.34n\rfloor$. Success rates of over $500$ trials are reported in   Figure   \ref{fig:succ-m-ex1-greedy}. The ascending trends of all lines demonstrate that the recoveries become easier when more and more samples are available, namely, when $m$ is rising. Once again, GPNP outperforms the others, achieving $75\%$ successful recoveries when $m/n=0.14$. By contrast, the best success rate among the other methods is $15\%$ from OMP when $m/n=0.14$. 

\begin{figure}[!th]
\begin{subfigure}{.48\textwidth}
	\centering
	\includegraphics[width=1\linewidth]{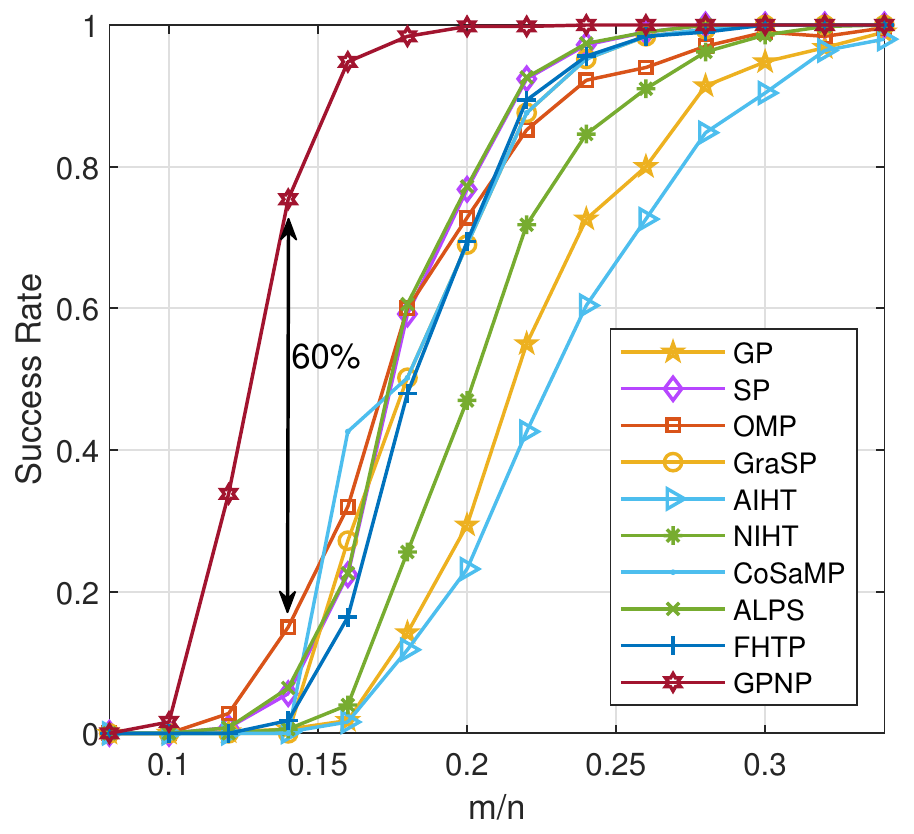}
	\caption{GPNP v.s. greedy methods.}
	\label{fig:succ-m-ex1-greedy}
\end{subfigure}	 
\begin{subfigure}{.48\textwidth}
	\centering
	\includegraphics[width=1\linewidth]{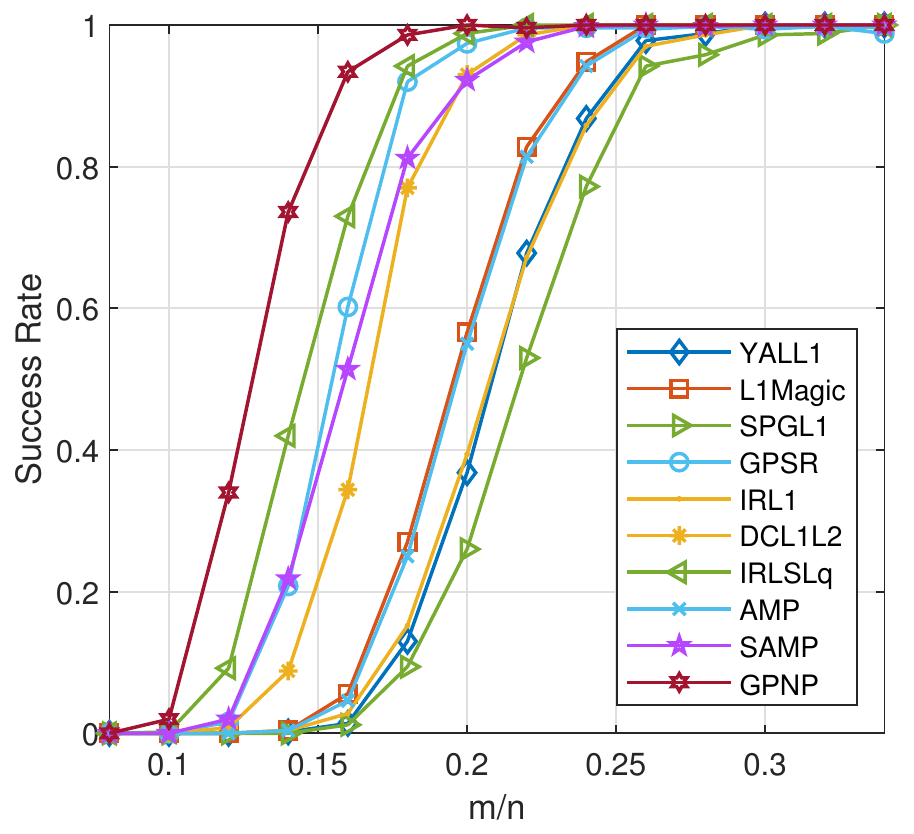}
	\caption{GPNP v.s. relaxation methods.}
	\label{fig:succ-m-ex1-non-greedy}
\end{subfigure}
\caption{Success rate v.s. sample size $m$ for Example \ref{cs-ex}.\label{fig:succ-m-ex1}}
\end{figure}
 
{\bf b) Relaxation methods  solving Example \ref{cs-ex}.} Next, we compare GPNP with nine relaxation methods. The results presented in  Figure   \ref{fig:succ-s-ex1-non-greedy} and  Figure  \ref{fig:succ-m-ex1-non-greedy} show that  GPNP outperforms these nine non-greedy methods regardless of varying $s$ or $m$.  However, in comparison with the gaps in  Figure   \ref{fig:succ-s-ex1-greedy}, the improvements in  Figure   \ref{fig:succ-s-ex1-non-greedy} are not very significant. This is because the non-greedy methods are known to have higher performance than greedy methods in terms of successful recovery rates. From  Figure    \ref{fig:succ-s-ex1-non-greedy} and  Figure  \ref{fig:succ-m-ex1-non-greedy}, one can conclude that methods (like IRL1, DCL1L2, IRLSLq and SAMP) based on more advanced relaxation models generate higher success rates than those (including YALL1, L1Magic, SPGL1, GPSR and AMP) based on the $\ell_1$-norm involved relaxation models.

{\bf c) All methods solving Example \ref{cs-ex} with higher dimensions.} Now, we would like to see the behaviours of all methods for solving CS problems with higher dimensional data. Average results over 20 trials are recorded in Table \ref{tab:cs-ind}, where $n$ is chosen from $\{10000,20000,30000\}$, $m=0.25\%n$ and $s=5\%n$. It can be clearly seen that the second-order methods SP, CoSaMP, SAMP and   GPNP obtain very tiny ReEr, namely,  much more accurate recoveries. When it comes to the computational speed, GPNP runs considerably fast, consuming 2.197 seconds when $n=30000$. In general, greedy methods run faster than relaxation ones.

{\bf d) Comparisons for solving Example \ref{cs-imag}.} This part fixates  CS problems with image data. As demonstrated above, relaxation methods run slowly in higher dimensional settings. Therefore, we will not include them in the subsequent comparisons. Moreover, the current implementations of OMP and GraSP are difficult to modify to solving  Example \ref{cs-imag} since measurement $\A$ is a function mapping instead of a  matrix. Hence, we exclude them in the following experiments as well. 

\begin{table}[!th]
	\renewcommand{\arraystretch}{1.05}\addtolength{\tabcolsep}{9.5pt}
	\caption{Effect of the bigger values of $n$ for Example \ref{cs-ex}.}\vspace{-2mm}
	\label{tab:cs-ind}
	\begin{center}
		\begin{tabular}{lccccccccc}
			\hline
 &\multicolumn{3}{c}{ReEr}&&\multicolumn{3}{c}{Time(seconds)}\\\cline{2-4}\cline{6-8}
${\rm Algs.}\setminus n$	&	$10000$	&	$20000$	&	$30000$	&	&	$10000$	&	$20000$	&	$30000$	\\\hline
GP	&	2.03e-04	&	8.93e-05	&	7.83e-05	&	&	14.88 	&	101.8 	&	340.6 	\\
SP	&	1.09e-15	&	1.57e-15	&	1.86e-15	&	&	0.311 	&	1.417	&	3.846	\\
OMP	&	1.39e-04	&	7.15e-05	&	6.61e-05	&	&	9.491 	&	59.48 	&	196.6 	\\
GraSP	&	1.87e-15	&	2.36e-15	&	8.65e-15	&	&	1.745 	&	7.844 	&	21.67 	\\
AIHT	&	8.86e-09	&	9.18e-09	&	1.04e-08	&	&	1.851 	&	7.194 	&	16.26 	\\
NIHT	&	3.13e-08	&	3.16e-08	&	3.09e-08	&	&	2.407 	&	10.34 	&	22.76 	\\
CoSaMP	&	4.98e-15	&	7.01e-15	&	8.64e-15	&	&	 0.435  	&	2.144 	&	5.857 	\\
ALPS	&	9.59e-09	&	1.28e-08	&	1.11e-08	&	&	2.214 	&	14.57 	&	52.80 	\\
FHTP	&	8.93e-09	&	1.13e-08	&	9.90e-09	&	&	1.408 	&	4.616 	&	10.69 	\\\hline
GPNP	&	1.23e-15	&	1.72e-15	&	1.99e-15	&	&	{ 0.228} 	&	{ 0.923} 	&	{  2.197} 	\\\hline
YALL1	&	3.42e-06	&	3.58e-06	&	3.68e-06	&	&	3.858 	&	18.43 	&	41.17 	\\
L1Magic	&	1.12e-05	&	1.74e-05	&	1.85e-05	&	&	31.96 	&	208.2 	&	615.0 	\\
SPGL1	&	9.68e-04	&	2.62e-05	&	1.39e-05	&	&	11.23 	&	46.23 	&	96.89 	\\
GPSR	&	7.14e-03	&	7.15e-03	&	7.57e-03	&	&	29.95 	&	139.7 	&	257.9 	\\
IRL1	&	9.21e-06	&	1.07e-05	&	1.03e-05	&	&	10.24 	&	44.54 	&	98.83 	\\
DCL1L2	&	3.11e-05	&	3.14e-05	&	3.29e-05	&	&	7.353 	&	30.88 	&	72.30 	\\
IRLSLq	&	4.91e-06	&	2.01e-05	&	2.77e-05	&	&	8.547 	&	57.18 	&	185.4 	\\
AMP	&	1.88e-06	&	1.92e-06	&	1.96e-06	&	&	4.073 	&	19.74 	&	40.35 	\\
SAMP	&	3.23e-15	&	4.05e-15	&	4.62e-15	&	&	38.47 	&	461.8 	&	2318 	\\
			\hline
		\end{tabular}
	\end{center}
\end{table}

\begin{table}[!th]
	\renewcommand{\arraystretch}{1.05}\addtolength{\tabcolsep}{-1pt}
	\caption{Results for Example \ref{cs-imag}.}\vspace{-2mm}
	\label{tab:ex3-cs}
	\begin{center}
		\begin{tabular}{lccccccccccccccc}
			\hline
&\multicolumn{7}{c}{Gray image}&&\multicolumn{7}{c}{Color image}	\\\cline{2-8}\cline{10-16}		
 &\multicolumn{3}{c}{PSNR}&&\multicolumn{3}{c}{Time(seconds)}
&&\multicolumn{3}{c}{PSNR}&&\multicolumn{3}{c}{Time(seconds)}\\
$s$	&	1000 	&	1500 	&	2000 	&	&	1000 	&	1500 	&	2000 	&	&	1000 	&	1500 	&	2000 	&	&	1000 	&	1500 	&	2000 	\\\cline{2-4}\cline{6-8}\cline{10-12}\cline{14-16}
& \multicolumn{15}{c}{${\tt nf}=0.05$}\\\hline
GP	&	19.31 	&	18.87 	&	18.20 	&	&	20.73 	&	31.14 	&	41.00 	&	&	27.34 	&	25.89 	&	24.46 	&	&	59.12 	&	88.49 	&	118.3 	\\
SP	&	19.19 	&	18.72 	&	18.08 	&	&	11.26 	&	10.33 	&	30.29 	&	&	26.32 	&	24.98 	&	24.07 	&	&	20.67 	&	25.22 	&	60.95 	\\
AIHT	&	20.15 	&	20.42 	&	20.18 	&	&	0.711 	&	1.776 	&	2.359 	&	&	29.25 	&	29.02 	&	27.83 	&	&	2.501 	&	4.851 	&	5.167 	\\
NIHT	&	20.00 	&	20.28 	&	19.92 	&	&	2.336 	&	2.640 	&	4.986 	&	&	29.01 	&	28.56 	&	27.43 	&	&	6.691 	&	6.262 	&	13.33 	\\
CoSaMP	&	19.24 	&	18.14 	&	17.62 	&	&	21.75 	&	23.92 	&	23.83 	&	&	25.72 	&	24.12 	&	23.21 	&	&	50.41 	&	69.80 	&	69.51 	\\
ALPS	&	19.21 	&	18.72 	&	18.18 	&	&	5.377 	&	13.93 	&	23.38 	&	&	26.37 	&	24.80 	&	24.20 	&	&	14.92 	&	36.56 	&	52.62 	\\
FHTP	&	20.15 	&	20.35 	&	20.12 	&	&	0.811 	&	1.332 	&	1.534 	&	&	29.24 	&	29.14 	&	27.93 	&	&	2.463 	&	3.927 	&	6.099 	\\
GPNP	&	20.16 	&	20.54 	&	20.32 	&	&	0.738 	&	1.629 	&	1.498 	&	&	29.40 	&	29.66 	&	28.65 	&	&	1.937 	&	2.934 	&	5.081 	\\\cline{2-16}
& \multicolumn{15}{c}{${\tt nf}=0.10$}\\\hline
GP	&	18.45 	&	17.48 	&	16.41 	&	&	20.86 	&	31.72 	&	43.15 	&	&	23.80 	&	21.25 	&	19.67 	&	&	58.41 	&	87.58 	&	115.1 	\\
SP	&	17.68 	&	16.59 	&	15.77 	&	&	7.434 	&	12.57 	&	23.25 	&	&	21.47 	&	19.81 	&	19.06 	&	&	19.15 	&	53.79 	&	50.79 	\\
AIHT	&	19.70 	&	19.45 	&	18.79 	&	&	0.766 	&	1.339 	&	1.864 	&	&	26.80 	&	24.66 	&	22.84 	&	&	3.197 	&	4.378 	&	5.695 	\\
NIHT	&	19.61 	&	19.10 	&	18.75 	&	&	1.502 	&	2.877 	&	3.308 	&	&	25.98 	&	23.91 	&	22.39 	&	&	5.295 	&	8.506 	&	11.80 	\\
CoSaMP	&	17.50 	&	16.07 	&	15.25 	&	&	11.74 	&	24.20 	&	24.42 	&	&	20.64 	&	19.09 	&	18.36 	&	&	34.80 	&	70.87 	&	69.17 	\\
ALPS	&	17.71 	&	16.57 	&	15.86 	&	&	7.245 	&	14.96 	&	22.25 	&	&	21.42 	&	20.00 	&	19.11 	&	&	19.95 	&	34.66 	&	51.21 	\\
FHTP	&	19.71 	&	19.64 	&	18.93 	&	&	0.981 	&	0.951 	&	1.489 	&	&	26.89 	&	24.93 	&	23.10 	&	&	2.882 	&	3.916 	&	4.638 	\\
GPNP	&	19.79 	&	19.87 	&	19.21 	&	&	0.640 	&	0.821 	&	1.602 	&	&	27.41 	&	25.80 	&	23.89 	&	&	2.323 	&	3.062 	&	6.471 	\\

\hline

		\end{tabular}
	\end{center}
\end{table}

Regarding the recovery accuracies, GPNP obtains the highest PSNR for all cases, which means it renders the most accurate recoveries, see  Figure  \ref{fig:gray-img} and Figure \ref{fig:color-img} as well  the data in Table \ref{tab:ex3-cs}.   It is well known that the first-order greedy methods have extremely high computational speed. Nevertheless,  GPNP runs the fastest for most scenarios, displaying its strong ability against other first-order methods.  We note that it runs much faster than the other two second-order methods SP and CoSaMP.

\begin{figure}[!th]
\includegraphics[width=1\textwidth]{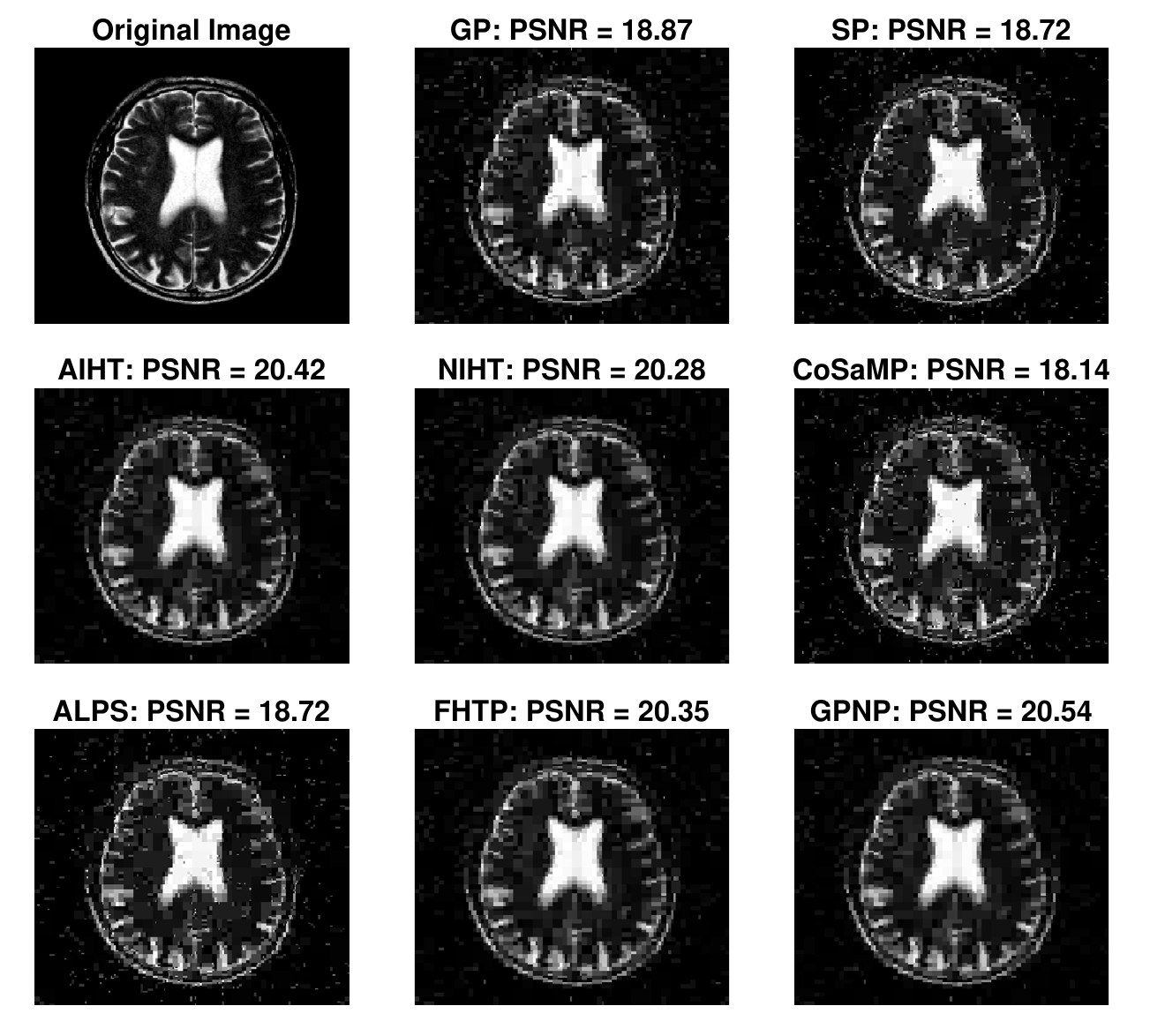} 
\vspace{-4mm}\caption{Recovery for Example \ref{cs-imag} with gray image data, where ${\tt nf}=0.05, s=1500$.\label{fig:gray-img}}
\end{figure}

\begin{figure}[!th]
\includegraphics[width=1\textwidth]{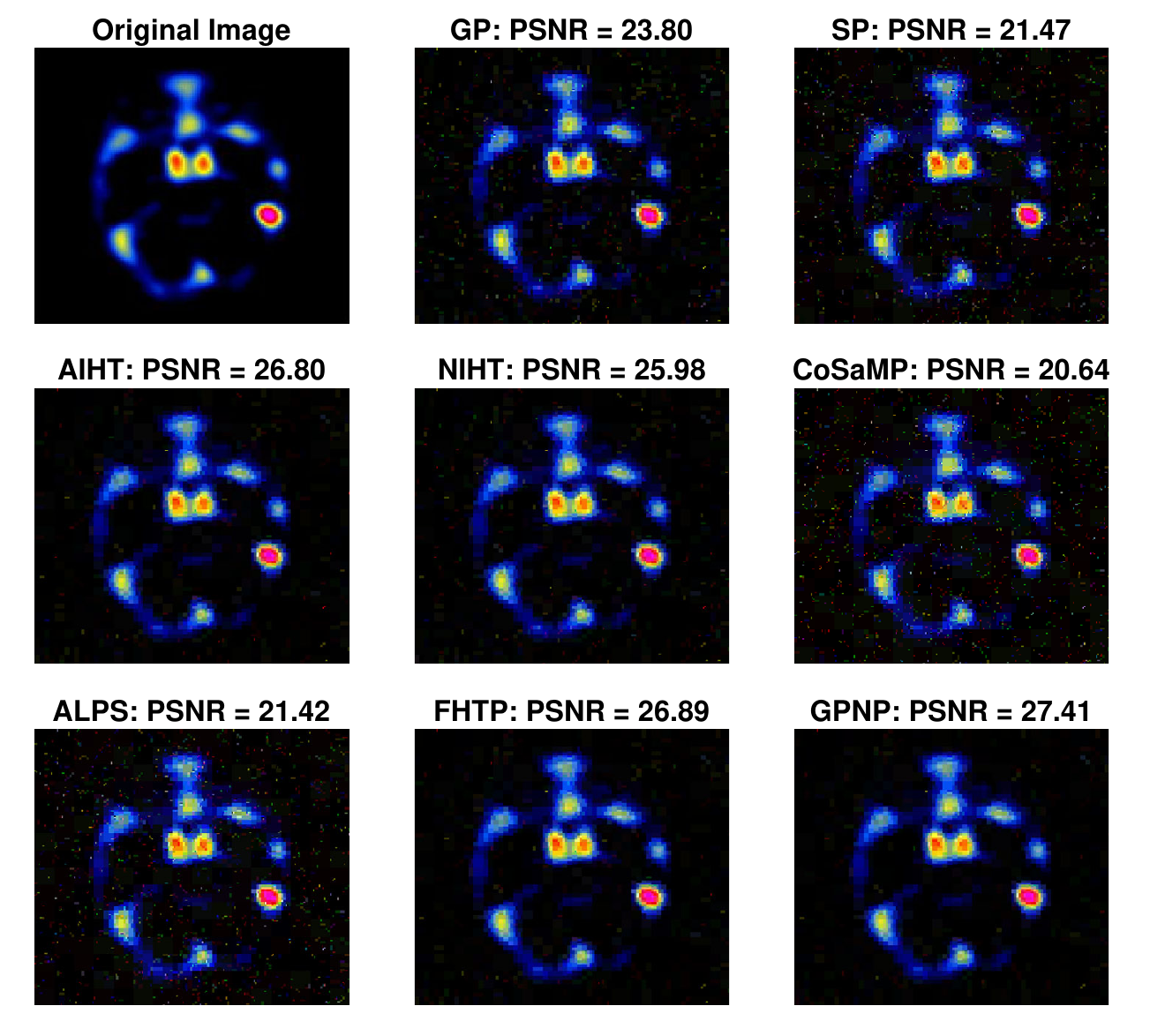} 
\vspace{-4mm}\caption{Recovery for Example \ref{cs-imag} with color image data, where ${\tt nf}=0.10, s=1000$.\label{fig:color-img}}
\end{figure}

\subsection{Solving QCS problems}
Quadratic compressive sensing (QCS) has been studied by \cite{shechtman2011sparsity,Beck13} and successfully applied into phase retrieval problems \cite{shechtman2014gespar}. The corresponding objective function in \eqref{sco} is the following quartic function,
$$\begin{array}{l}
f_{qcs}(\bx):=\frac{1}{4m}\sum_{i\in[m]}\left(\langle \bx, \A_i \bx\rangle -b_i \right)^2,
\end{array}$$
where $\A_i\in\R^{n\times n}, i\in[m]$ and $\bb\in\R^m$. For such an example, similar to  \cite{Beck13}, we consider $\A_i=\ba_i\ba_i^\top$ with $\ba_i\in\R^n$ and $a_{ij}\sim \mathcal{N}(0,1), i\in[m],j\in[n]$. The `ground truth' $s$-sparse solution $\bx^*$ is generated the same as Example \ref{cs-ex} and  $\bb$ is obtained by $b_i=\langle \ba_i, \bx^*\rangle^2,i\in[m]$.  We compare GPNP with two methods RSS and GSS proposed in \cite{Beck13}.  Firstly, by fixing $m=80, n=120$ and altering $s$ from $\{3,4,\ldots,15\}$, we run 100 trials and report the number of successful recoveries. Results   in Table \ref{tab:ex4-qcs} show the better recovery ability of GPNP than the other two methods. We next compare them for solving QCS with higher dimensions. For each $n$ form $\{1000,2000,\ldots,5000\}$, we set $m=0.8n$ and $s=0.01n$ and record the average results over 20 trials (Here, only results of successful recoveries are reported). It can be clearly seen from Table \ref{tab:ex4-qcs-1} that GPNP runs the fastest and delivers the smallest objective function values.

\begin{table}[!th]
	\renewcommand{\arraystretch}{1.0}\addtolength{\tabcolsep}{4pt}
	\caption{The number of successful recoveries.}\vspace{-2mm}
	\label{tab:ex4-qcs}
	\begin{center}
		\begin{tabular}{lccccccccccccc}
			\hline
$s$	&	3 	&	4 	&	5 	&	6 	&	7 	&	8 	&	9 	&	10 	&	11 	&	12 	&	13 	&	14 	&	15 	\\\hline
PSS	&	42 	&	46 	&	32 	&	38 	&	37 	&	29 	&	20 	&	26 	&	13 	&	14 	&	11 	&	8 	&	0 	\\
GSS	&	90 	&	95 	&	94 	&	86 	&	68 	&	71 	&	59 	&	45 	&	38 	&	34 	&	20 	&	19 	&	8 	\\
GPNP	&	93 	&	98 	&	98 	&	100 	&	100 	&	100 	&	98 	&	100 	&	96 	&	99 	&	91 	&	86 	&	70 	\\

\hline
		\end{tabular}
	\end{center}
\end{table}

\begin{table}[!th]
	\renewcommand{\arraystretch}{1.0}\addtolength{\tabcolsep}{4pt}
	\caption{Effect of the bigger values of $n$ for QCS problems.}\vspace{-2mm}
	\label{tab:ex4-qcs-1}
	\begin{center}
		\begin{tabular}{l ccccccccc}
			\hline
&\multicolumn{4}{c}{$f_{qcs}$}&&\multicolumn{4}{c}{Time(in seconds)}\\\cline{2-5}\cline{7-10}
$n$	&	1000 	&	2000 	&	3000 	&	4000 	&&	1000 	&	2000 	&	3000 	&	4000 	\\\hline
PSS	&	1.12e-12	&	2.88e-12	&	2.80e-12	&	3.00e-12	&&	2.450 	&	19.38	&	56.40 	&	135.0 	\\
GSS	&	8.98e-13	&	2.53e-12	&	2.84e-12	&	2.94e-12	&&	8.951 	&	157.0 	&	629.4 	&	2227 	\\
GPNP	&	2.74e-18	&	1.02e-16	&	1.40e-16	&	1.91e-16	&&	0.004 	&	0.101 	&	0.258 	&	0.933 	\\
\hline
		\end{tabular}
	\end{center}
\end{table}

\section{Conclusion}\label{sec:conclusion}
The proposed algorithm combines the hard-thresholding operator and  Newton pursuit, leading to a low computational complexity and fast convergence speed. Its high performance has been demonstrated by extensive numerical experiments for solving compressive sensing and quadratic compressive sensing problems in comparison with a number of excellent solvers.  To establish the global and quadratic convergence properties, we made use of the popular regularities of the objective function. The proofs were quite standard and can be regarded as the extension of Newton-type methods for solving unconstrained optimization problems.  We feel that the techniques used to cast the algorithm and to do the convergence analysis might be helpful for dealing with the sparsity constrained optimization with equalities or inequalities constraints, which is left for future research.





\bibliographystyle{elsarticle-num}
\bibliography{references}







\end{document}